\documentclass[12pt, twoside]{amsart}
\usepackage{amsmath, amsthm, amssymb}
\usepackage{verbatim} 
\usepackage[headings]{fullpage}

 \numberwithin{equation}{section}
 
 \theoremstyle{plain}
 \newtheorem{theorem}{Theorem}[section]
 \newtheorem{corollary}[theorem]{Corollary}
 \newtheorem{lemma}[theorem]{Lemma}
 
 \newtheorem{proposition}[theorem]{Proposition}
 \newtheorem{question}[theorem]{Question}
 
 \theoremstyle{definition}
 \newtheorem{definition}[theorem]{Definition}
 \newtheorem{definitions}[theorem]{Definitions}
 
 \newtheorem{example}[theorem]{Example}
 \newtheorem{examples}[theorem]{Examples} 
 \newtheorem{remark}[theorem]{Remark}
 \newtheorem{remarks}[theorem]{Remarks}
 
 \DeclareMathOperator{\im}{im}
 \DeclareMathOperator{\id}{id}
 \DeclareMathOperator{\ann}{ann}
 \DeclareMathOperator{\Ann}{Ann}
 \DeclareMathOperator{\soc}{soc}
 \DeclareMathOperator{\Spec}{Spec}
 \DeclareMathOperator{\Max}{Max}
 \DeclareMathOperator{\Ext}{Ext}
 
 \DeclareMathOperator{\Kdim}{K.\mspace{1mu}dim}
 \DeclareMathOperator{\rKdim}{r.\mspace{1mu}K.\mspace{1mu}dim}
 \DeclareMathOperator{\lKdim}{l.\mspace{1mu}K.\mspace{1mu}dim}
 \DeclareMathOperator{\rgldim}{r.\mspace{1mu}gl.\mspace{1mu}dim}
 \DeclareMathOperator{\pd}{pd}
 
 \DeclareMathOperator{\characteristic}{char}
 \DeclareMathOperator{\rad}{rad}
 
 \newcommand{\F}{\mathcal{F}}
 \newcommand{\Fpr}{\mathcal{F}_{\textnormal{pr}}}
 \newcommand{\C}{\mathcal{C}}
 
 \newcommand{\setS}{\mathcal{S}}
 \newcommand{\M}{\mathfrak{M}}
 \newcommand{\m}{\mathfrak{m}}

 \newcommand{\Fprsim}{\mathcal{F}_{\textnormal{pr}}^{\mspace{2mu}\circ}}
 
 \newcommand{\separate}{\medskip}

\begin{document}

\title[Noncommutative generalizations of theorems of Cohen and Kaplansky]
{Noncommutative generalizations of\\
theorems of Cohen and Kaplansky}
\author{Manuel L. Reyes}
\address{Department of Mathematics\\
University of California\\
Berkeley, CA, 94720-3840}
\curraddr{Department of Mathematics\\
University of California, San Diego\\
9500 Gilman Drive, \#0112\\
La Jolla, CA  92093-0112}
\email{m1reyes@math.ucsd.edu}
\urladdr{http://math.ucsd.edu/~m1reyes/}

\keywords{point annihilators, cocritical right ideals, Cohen's Theorem,
right noetherian rings, Kaplansky's Theorem, principal right ideals, Krull dimension,
semiprime rings, essential right ideals}
\subjclass[2010]{Primary: 16D25, 16P40, 16P60; Secondary: 16N60}
\date{June 2, 2011}

\begin{abstract}
This paper investigates situations where a property of a ring can be tested on a set
of ``prime right ideals.'' 
Generalizing theorems of Cohen and Kaplansky, we show that every right ideal of
a ring is finitely generated (resp.\ principal) iff every ``prime right ideal'' is finitely
generated (resp.\ principal), where the phrase ``prime right ideal'' can be interpreted
in one of many different ways.
We also use our methods to show that other properties can be tested on special
sets of right ideals, such as the right artinian property and various homological properties.
Applying these methods, we prove the following noncommutative generalization of a result
of Kaplansky: a (left and right) noetherian ring is a principal right ideal ring
iff all of its maximal right ideals are principal.
A counterexample shows that the left noetherian hypothesis cannot be dropped.
Finally, we compare our results to earlier generalizations of Cohen's and Kaplansky's
theorems in the literature.
\keywords{point annihilators \and cocritical right ideals \and Cohen's Theorem \and
right noetherian rings \and Kaplansky's Theorem \and principal right ideals}
\end{abstract}

\maketitle

\section{Introduction}

Two famous theorems from commutative algebra due to I.\,S.~Cohen and I.~Kaplansky state
that, to check whether every ideal in a commutative ring is finitely generated
(respectively, principal), it suffices to test only the prime ideals. Cohen's Theorem
appeared as Theorem~2 of~\cite{Cohen}.

\begin{theorem}[Cohen's Theorem]
\label{original Cohen's Theorem}
A commutative ring $R$ is noetherian iff every prime ideal of $R$ is finitely
generated.
\end{theorem}

Also, we recall a characterization of commutative principal ideal rings due to
I.~Kaplansky, which appeared as Theorem~12.3 of~\cite{KapDivisors}. Throughout this
paper, a ring in which all right ideals are principal will be called a \emph{principal right ideal ring},
or PRIR. Similarly, we have principal left ideal rings (PLIRs), and a ring which is
both a PRIR and a PLIR is called a \emph{principal ideal ring}, or PIR. 

\begin{theorem}[Kaplansky's Theorem]
\label{original Kap's Theorem}
A commutative noetherian ring $R$ is a principal ideal ring iff every maximal
ideal of $R$ is principal.
\end{theorem}

Combining this result with Cohen's Theorem, Kaplansky deduced the following in
Footnote~8 on p.~486 of~\cite{KapDivisors}.

\begin{theorem}[Kaplansky-Cohen Theorem]
\label{original Kap Cohen Theorem}
A commutative ring $R$ is a principal ideal ring iff every prime ideal of $R$ is
principal.
\end{theorem}

(We refer to this result as the Kaplansky-Cohen Theorem for two reasons. The primary
and most obvious reason is that it follows from a combination of the above results
due to Cohen and Kaplansky. But we also use this term because it is a result in the
spirit of Cohen's Theorem, that was first deduced by Kaplansky.)

The unifying theme of this paper is the generalization of the above theorems to
noncommutative rings, using certain families of right ideals as our tools.
Let us mention a typical method of proof of these theorems, as this
will guide our investigation into their noncommutative generalizations. One first
assumes that the prime ideals of a commutative ring $R$ are all finitely generated
(or principal), but that there exists an ideal of $R$ that is not f.g.\ (or principal).
One then passes to a ``maximal counterexample'' by Zorn's Lemma and proves
that such a maximal counterexample is prime. This contradicts the assumption that
all primes have the relevant property, proving the theorem. 

These ``maximal implies prime'' theorems were systematically studied in~\cite{LR}
from the viewpoint of certain families of ideals, called \emph{Oka families}. In
particular, the Cohen and Kaplansky-Cohen theorems were recovered in~\cite[p.~3017]{LR}.
These families were generalized to noncommutative rings in~\cite{Reyes}, where we
defined \emph{Oka families of right ideals}. This resulted in a noncommutative
generalization of Cohen's Theorem in~\cite[Thm.~6.2]{Reyes}, stating that a ring
$R$ is right noetherian iff its \emph{completely prime right ideals} are all
finitely generated. In the present paper, we will improve upon this result, providing
smaller ``test sets'' of right ideals that need to be checked to determine if a ring is right noetherian.
In addition, we will provide generalizations of Kaplansky's Theorem~\ref{original Kap's Theorem}
and the Kaplansky-Cohen Theorem~\ref{original Kap Cohen Theorem}. 

We begin by reviewing the relevant results from~\cite{LR} in~\S\ref{review section}.
This includes an introduction to the notions of right Oka families, classes of cyclic
modules closed under extensions, completely prime right ideals, the Completely Prime
Ideal Principle (CPIP), and the CPIP Supplement.
Our work in \S\S\ref{point annihilator section}--\ref{point annihilator theorem section}
addresses the following question: what are some sufficient conditions for all right ideals
of a ring to lie in a given right Oka family? In~\S\ref{point annihilator section}
we develop the idea of a \emph{(noetherian) point annihilator set} in order to deal with this problem.
Then in~\S\ref{point annihilator theorem section} we prove the \emph{Point Annihilator
Set Theorem~\ref{first point annihilator set theorem}}.
Along with its consequences, such as Theorem~\ref{"PIP supplement"}, this theorem
gives sufficient conditions for a property of right ideals to be testable on a smaller set
of right ideals.
We achieve a generalization of Cohen's Theorem in Theorem~\ref{Cohen's Theorem}. This result
is ``flexible'' in the sense that, in order to check whether a ring is right noetherian, one
can use various test sets of right ideals (in fact, certain point annihilator sets will work).
One important such set is the \emph{cocritical right ideals}.
Other consequences of the Point Annihilator Set Theorem are also investigated.

Next we consider families of principal right ideals in~\S\ref{principal right ideal section}.
Whereas the family of principal ideals of a commutative ring is always an
Oka family, it turns out that the family $\Fpr$ of principal right ideals can
fail to be a right Oka family in certain noncommutative rings.
By defining a right Oka family $\Fprsim$ that ``approximates'' $\Fpr$, we are able to
provide a noncommutative generalization of the Kaplansky-Cohen Theorem in Theorem~\ref{Kap}.
As before, a specific version of this theorem is the following: \emph{a ring is a principal
right ideal ring iff all of its cocritical right ideals are principal.}

In~\S\ref{essential right ideal section} we sharpen our versions of the Cohen
and Kaplansky-Cohen Theorems by considering families of right ideals that are
closed under direct summands. This allows us to reduce the ``test sets'' of the
Point Annihilator Set Theorem~\ref{"PIP supplement"} to sets of essential right
ideals. For instance, to check if a ring is right noetherian or a principal right
ideal ring, it suffices to test the \emph{essential} cocritical right ideals.  
Other applications involving homological properties of right ideals are considered.

We work toward a noncommutative generalization of Kaplansky's
Theorem~\ref{original Kap's Theorem} in~\S\ref{Kap's theorem section}.
The main result here is Theorem~\ref{noetherian Kap}, which states that
\emph{a (left and right) noetherian ring is a principal right ideal ring iff its maximal
right ideals are principal}.
Notably, our analysis also implies that such a ring has right Krull dimension~$\leq 1$.
An example shows that the theorem does not hold if the left noetherian hypothesis
is omitted.

Finally, we explore the connections between our results and previous generalizations
of the Cohen and Kaplansky-Cohen theorems in \S\ref{previous generalizations section}.
These results include theorems due to V.\,R.~Chandran, K.~Koh, G.\,O.~Michler, 
P.\,F.~Smith, and B.\,V.~Zabavs'ki\u{\i}. Discussing these earlier results affords us an
opportunity to survey some previous notions of ``prime right ideals'' studied in the literature.

\subsection*{Conventions}

All rings are associative and have identity, and all modules and ring homomorphisms
are unital. Let $R$ be a ring. We say $R$ is a \emph{semisimple ring} if $R_R$
is a semisimple module. We denote the Jacobson radical of $R$ by $\rad(R)$.
We say that $R$ is \emph{semilocal} (resp.\ \emph{local}) if $R/\rad(R)$ is
semisimple (resp.\ a division ring).
Given a family $\F$ of right ideals in a ring $R$, we let $\F'$ denote the
complement of $\F$ within the set of all right ideals of $R$. 
Now fix an $R$-module $M_R$. We will write $N\subseteq_{e}M$ to mean that $N$ is
an essential submodule of $M$. A \emph{proper factor} of $M$ is a module of the
form $M/N$ for some nonzero submodule $N_R\subseteq M$.

\section{Review of right Oka families}
\label{review section}

In~\cite{Reyes}, we introduced the following notion of a ``one-sided prime.''

\begin{definition}
A right ideal $P_R\subsetneq R$ is a \emph{completely prime right ideal}
if, for all $a,b\in R$,
\[
aP \subseteq P \text{ and } ab \in P \implies a \in P \text{ or } b \in P.
\]
\end{definition}

Notice immediately that a two-sided ideal is completely prime as a right ideal iff it is a
completely prime ideal (that is, the factor ring is a domain).
In particular, the completely prime right ideals of a commutative ring are precisely the
prime ideals of that ring.

One way in which these right ideals behave like prime ideals of commutative rings
is that \emph{right ideals that are maximal in certain senses tend to be completely
prime}. A more precise statement requires a definition.
Given a right ideal $I$ and element $a$ of a ring $R$, we denote
\[
a^{-1}I = \{ r \in R : ar \in I \}.
\]

\begin{definition}
A family $\F$ of right ideals in a ring $R$ is an \emph{Oka family of right ideals} (or
a \emph{right Oka family}) if $R\in F$ and, for any element $a\in R$ and any right ideal
$I_R\subseteq R$,
\[
I + aR,\ a^{-1}I \in \F \implies I \in \F.
\]
\end{definition}

For a family $\F$ of right ideals in a ring $R$, we let $\F'$ denote the \emph{complement}
of $\F$ (the set of right ideals of $R$ that do not lie in $\F$), and we let $\Max(\F')$ denote
the set of right ideals of $R$ that are maximal in $\F'$.
The precise ``maximal implies prime'' result, which was proved in~\cite[Thm.~3.6]{Reyes}, can
now be stated.

\begin{theorem}[Completely Prime Ideal Principle]
\label{CPIP}
Let $\F$ be an Oka family of right ideals in a ring $R$. Then any right ideal
$P\in\Max(\F')$ is completely prime.
\end{theorem}

A result accompanying the Completely Prime Ideal Principle (CPIP) shows that, for
special choices of right Oka families $\F$, in order to test whether $\F$ consists of all right
ideals it is enough to check that all completely prime right ideals lie in $\F$. Throughout
this paper, a family $\F_0$ of right ideals in a ring $R$ is called a \emph{semifilter} if,
whenever $J\in\F$ and $I$ is a right ideal of $R$, $J\subseteq I$ implies $I\in\F$.

\begin{theorem}[Completely Prime Ideal Principle Supplement]
\label{CPIP supplement}
Let $\F$ be a right Oka family in a ring $R$ such that every nonempty chain of right
ideals in $\F'$ (with respect to inclusion) has an upper bound in $\F'$. Let $\setS$
denote the set of completely prime right ideals of $R$.
\begin{itemize}
\item Let $\F_0$ be a semifilter of right ideals in $R$. If $\setS\cap\F_0\subseteq\F$,
then $\F_0\subseteq\F$.
\item For $J_R\subseteq R$, if all right ideals in $\setS$ containing $J$ (resp.\ properly
containing $J$) belong to $\F$, then all right ideals containing $J$ (resp.\ properly
containing $J$) belong to $\F$.
\item If $\setS\subseteq\F$, then $\F$ consists of all right ideals of $R$.
\end{itemize}
\end{theorem}

In order to efficiently construct right Oka families, we established the following
correspondence in~\cite[Thm.~4.7]{Reyes}. A class $\C$ of cyclic right $R$-modules is
said to be \emph{closed under extensions} if $0\in\C$ and, for every short exact
sequence $0\to L\to M\to N\to 0$ of \underline{cyclic} right $R$-modules, $L\in\C$
and $N\in\C$ imply $M\in\C$. Given a class of cyclic right $R$-modules, one may
construct the following family of right ideals of $R$:
\[
\F_{\C} := \{I_R\subseteq R : R/I\in\C \}.
\]
Conversely, given a family $\F$ of right ideals in $R$, we construct a class
of cyclic $R$-modules
\[
\C_{\F} := \{ M_R : M\cong R/I \text{ for some } I \in \F \}.
\]

\begin{theorem}
\label{cyclic module correspondence}
Given a class $\C$ of cyclic right $R$-modules that is closed under extensions, the
family $\F_{\C}$ is a right Oka family. Conversely, given a right Oka family $\F$, the
class $\C_{\F}$ is closed under extensions.
\end{theorem}

This theorem was used to construct a number of examples of right Oka families
in~\cite{Reyes}. For us, the most important such example is the finitely generated
right ideals: \emph{in any ring $R$, the family of finitely generated right ideals
of $R$ is a right Oka family} (see~\cite[Prop.~3.7]{Reyes}).

An easy consequence of the above theorem, proved in~\cite[Cor.~4.9]{Reyes}, will be
useful throughout this paper.

\begin{corollary}
\label{repeated extensions}
Let $\F$ be a right Oka family in a ring $R$. Suppose that $I_R\subseteq R$ is such
that the right $R$-module $R/I$ has a filtration
\[
0 = M_0 \subseteq M_1 \subseteq \cdots \subseteq M_n = R/I
\]
where each filtration factor is cyclic and of the form $M_j / M_{j-1} \cong R/I_j$ for
some $I_j \in \F$. Then $I\in\F$.
\end{corollary}

An important issue to be dealt with in the proof of Theorem~\ref{cyclic module correspondence}
is that of similarity of right ideals. Two right ideals $I$ and $J$ of a ring $R$ are
\emph{similar}, written $I\sim J$, if $R/I\cong R/J$ as right $R$-modules. Two results
from~\cite[\S4]{Reyes} about isomorphic cyclic modules will be relevant to the
present paper, the second of which deals directly with similarity of right ideals.

\begin{lemma}
\label{isomorphic cyclic modules}
For any ring $R$ with right ideal $I\subseteq R$ and element $a\in R$, the following
cyclic right $R$-modules are isomorphic: $R/a^{-1}I\cong (I+aR)/I$.
\end{lemma}

\begin{proposition}
\label{Oka is similarity closed}
Every right Oka family $\F$ in a ring $R$ is closed under similarity: if $I\in\F$ then every
right ideal similar to $I$ is also in $\F$.
\end{proposition}

Finally, in~\cite[\S6]{Reyes} we studied a special collection of completely prime right ideals.

\begin{definition}
A module $M\neq 0$ is monoform if, for every nonzero submodule $N_R\subseteq M$, every
nonzero homomorphism $N\to M$ is injective.
A right ideal $P_R\subsetneq R$ is \emph{comonoform} if the right $R$-module $R/P$ is
monoform.
\end{definition}

It was shown in~\cite[Prop.~6.3]{Reyes} that \emph{every comonoform right ideal is
completely prime}. The idea is that comonoform right ideals form an especially ``well-behaved''
subset of the completely prime right ideals of any ring.

\section{Point annihilator sets for classes of modules}
\label{point annihilator section}

In this section we develop an appropriate notion of a ``test set'' for certain
properties of right ideals in noncommutative rings. This is required for us
to state the main theorems along these lines in the next section.
Recall that a \emph{point annihilator} of a module $M_R$
is defined to be an annihilator of a nonzero element $m\in M\setminus\{0\}$.

\begin{definitions}
Let $\C$ be a class of right modules over a ring $R$. A set $\setS$ of
right ideals of $R$ is a \emph{point annihilator set for $\C$} if every
nonzero $M\in\C$ has a point annihilator that lies in $\setS$. In
addition, we make the following two definitions for special choices of
$\C$:
\begin{itemize}
\item A point annihilator set for the class of all right $R$-modules
will simply be called a \emph{(right) point annihilator set for $R$}.
\item A point annihilator set for the class of all noetherian right
$R$-modules will be called a \emph{(right) noetherian point annihilator
set for $R$}.
\end{itemize}
\end{definitions}

Notice that a point annihilator set need not contain the unit
ideal $R$, because point annihilators are always proper right ideals.
Another immediate observation is that, for a right noetherian ring
$R$, a right point annihilator set for $R$ is the same as a right
noetherian point annihilator set for $R$.

\begin{remark}
\label{larger point annihilator sets}
The idea of a point annihilator set $\setS$ for a class of modules $\C$ is
simply that $\setS$ is ``large enough'' to contain a point annihilator of every
nonzero module in $\C$. In particular, our definition \emph{does not} require
every right ideal in $\setS$ to actually be a point annihilator for some module
in $\C$. This means that any other set $\setS '$ of right ideals with
$\setS '\supseteq \setS$ is also a point annihilator set for $\C$. On the other
hand, if $\C_0 \subseteq \C$ is a subclass of modules, then $\setS$ is again a
point annihilator set for $\C_0$.
\end{remark}

\begin{remark}
\label{point annihilator set restated}
Notice that $\setS$ is a point annihilator set for a class $\C$ of modules iff, for
every nonzero module $M_R\in\C$, there exists a \emph{proper} right ideal
$I\in\setS$ such that the right module $R/I$ embeds into $M$.
\end{remark}

The next result shows that noetherian point annihilator sets for a ring $R$ 
exert a considerable amount of control over the noetherian right $R$-modules.

\begin{lemma}
A set $\setS$ of right ideals in $R$ is a noetherian point annihilator
set iff for every noetherian module $M_R\neq 0$, there is a finite
filtration of $M$
\[
0 = M_0 \subsetneq M_1 \subsetneq \cdots \subsetneq M_n = M
\]
such that, for $1\leq j\leq n$, there exists $I_j\in\setS$ such that
$M_j/M_{j-1}\cong R/I_j$. 
\end{lemma}

\begin{proof}
The ``if'' direction is easy, so we will prove the ``only if'' part.
For convenience, we will refer to a filtration like the one described
above as as an \emph{$\setS$-filtration}. Suppose that $\setS$ is a noetherian
point annihilator set for $R$, and let $M_R \neq 0$ be noetherian. We
prove by noetherian induction that $M$ has an $\setS$-filtration. Consider
the set $\mathcal{X}$ of nonzero  submodules of $M$ that have an
$\setS$-filtration. Because $\setS$ is a noetherian point annihilator set,
it follows that $\mathcal{X}$ is nonempty. Since $M$ is noetherian,
$\mathcal{X}$ has a maximal element, say $N$. Assume for contradiction
that $N\neq M$. Then $M/N\neq 0$ is noetherian, and by hypothesis there
exists $I\in\setS$ with $I\neq R$ such that $R/I\cong N'/N\subseteq M/N$ for some
$N'_R\subseteq M$. But then $N\subsetneq N'\in\mathcal{X}$, contradicting
the maximality of $N$. Hence $M=N\in\mathcal{X}$, completing the proof.
\end{proof}

We wish to highlight a special type of point annihilator set in the
definition below.

\begin{definition}
\label{closed point annihilator set definition}
A set $\setS$ of right ideals of a ring $R$ is \emph{closed under point
annihilators} if, for all $I\in\setS$, every point annihilator of $R/I$
lies in $\setS$. (This is equivalent to saying that $I\in\setS$ and
$x\in R\setminus I$ imply $x^{-1}I\in \setS$.) If $\C$ is a class of right
$R$-modules, we will say that $\setS$ is a \emph{closed point annihilator
set for $\C$} if $\setS$ is a point annihilator set for $\C$ and $\setS$
is closed under point annihilators.
\end{definition}

The idea of the above definition is that $\setS$ is ``closed under
passing to further point annihilators of $R/I$'' whenever $I\in\setS$.
The significance of these closed point annihilator sets is demonstrated
by the next result.

\begin{lemma}
\label{closed point annihilator set}
Let $\C$ be a class of right modules over a ring $R$ that is closed under
taking submodules (e.g.\ the class of noetherian modules). Suppose that
$\setS$ is a closed point annihilator set for $\C$. Then for any other
point annihilator set $\setS_1$ of $\C$, the set $\setS_1\cap\setS$ is a
point annihilator set for $\C$.
\end{lemma}

\begin{proof}
Let $0\neq M_R\in\C$. Because $\setS$ is a point annihilator set for $\C$, 
there exists $0\neq m\in M$ such that $I:=\ann(m)\in\setS$. By the
hypothesis on $\C$, the module $mR$ lies in $\C$. Because $\setS_1$ is
also a point annihilator set for $\C$, there exists $0\neq mr\in mR$ such
that  $\ann(mr)\in\setS_1$. The fact that $\setS$ is closed implies that
$\ann(mr)\in\setS\cap\setS_1$.  This proves that $\setS\cap\setS_1$ is
also a point annihilator set for $\C$.
\end{proof}

The prototypical example of a noetherian point annihilator set is the
prime spectrum of a commutative ring. In fact, \emph{every} noetherian
point annihilator set in a commutative ring can be ``reduced to'' some
set of prime ideals, as we show below.

\begin{proposition}
\label{commutative point annihilator set}
In any commutative ring $R$, the set $\Spec (R)$ of prime ideals is a
closed noetherian point annihilator set. Moreover, given any noetherian
point annihilator set $\setS$ for $R$, the set $\setS\cap\Spec(R)$ is a
noetherian point annihilator subset of $S$ consisting of prime ideals.
\end{proposition}

\begin{proof}
The set $\Spec (R)$ is a noetherian point annihilator set thanks to the
standard fact that any noetherian module over a commutative ring has an
associated prime; see, for example,~\cite[Thm.~3.1]{Eisenbud}. Furthermore,
this set is closed because for $P\in\Spec (R)$, the annihilator of any
nonzero element of $R/P$ is equal to $P$. The last statement now follows
from Lemma~\ref{closed point annihilator set}.
\end{proof}

In this sense right noetherian point annihilator sets of a ring generalize
the concept of the prime spectrum of a commutative ring.
However, one should not push this analogy too far: in a commutative ring
$R$, \emph{any} set $\setS$ of ideals containing $\Spec(R)$
is also a noetherian point annihilator set! In fact, with the help of
Proposition~\ref{commutative point annihilator set} it is easy to verify
that any commutative ring $R$ has smallest noetherian point annihilator
set $\setS_0:=\{P\in\Spec(R):R/P\text{ is noetherian}\}$, and that a set
$\setS$ of ideals of $R$ is a noetherian point annihilator set for $R$
iff $\setS\supseteq\setS_0$.

\separate

For most of the remainder of this section, we will record a number of
examples of point annihilator sets that will be useful in later applications.
Perhaps the easiest example is the following: the family of all right ideals
of a ring $R$ is a point annihilator set for any class of right $R$-modules.
A less trivial example: the family of maximal right ideals of a ring $R$ is a
point annihilator set for the class of right $R$-modules of finite length, or
for the larger class of artinian right modules. More specifically, according
to Remark~\ref{point annihilator set restated} it suffices to take any set
$\{\m_i\}$ of maximal right ideals such that the $R/\m_i$ exhaust all
isomorphism classes of simple right modules.

\begin{example}
\label{semi-artinian example}
Recall that a module $M_R$ is said to be semi-artinian if every
nonzero factor module of $M$ has nonzero socle, and that a ring $R$ is
right semi-artinian if $R_R$ is a semi-artinian module. One can readily verify
that $R$ is right semi-artinian iff every nonzero right $R$-module has nonzero
socle. Thus for such a ring $R$, the set of maximal right ideals is a point annihilator
set for $R$, and in particular it is a noetherian point annihilator set
for $R$.
\end{example}

\begin{example}
\label{left perfect example}
Let $R$ be a left perfect ring, that is, a semilocal ring
whose Jacobson radical is left $T$-nilpotent---see~\cite[\S23]{FC} for
details. (Notice that this class of rings includes semiprimary rings,
especially right \emph{or} left artinian rings.) By a theorem of Bass (see~\cite[(23.20)]{FC}),
over such a ring, every \emph{right} $R$-module satisfies DCC on cyclic submodules.
Thus every nonzero right module has nonzero socle, and such a ring is
right semi-artinian. But $R$ has finitely many simple modules up to
isomorphism (because the same is true modulo its Jacobson radical).
Choosing  a set $\setS =\{\m_1,\dots,\m_n\}$ of maximal right ideals such
that the modules $R/\m_i$ exhaust the isomorphism classes of simple right
$R$-modules, we conclude by Remark~\ref{point annihilator set restated}
that $\setS$ is a point annihilator set for any class of right modules
$\C$. Hence $\setS$ forms a right noetherian point annihilator set for
$R$. (The observant reader will likely have noticed that the same argument
applies more generally to any right semi-artinian ring with finitely many
isomorphism classes of simple right modules.)
\end{example}

Directly generalizing the fact that the prime spectrum of a commutative
ring is a noetherian point annihilator set, we have the following fact,
valid for any noncommutative ring.

\begin{proposition}
\label{completely prime NPA}
The set of completely prime right ideals in any ring $R$ is a noetherian
point annihilator set.
\end{proposition}

\begin{proof}
Let $M_R\neq 0$ be noetherian. For any point annihilator $I=\ann(m)$
with $0\neq m\in M$, the module $R/I\hookrightarrow M$ is noetherian.
Thus $M$ must have a \emph{maximal} point annihilator $P_R\supseteq I$,
and $P$ is completely prime by~\cite[Prop.~5.3]{Reyes}. 
\end{proof}

Recall that in any ring $R$, the set of comonoform right ideals of $R$ forms
a subset of the set of all completely prime right ideals of $R$. As we show
next, the subset of comonoform right ideals is also a noetherian point
annihilator set.

\begin{proposition}
\label{comonoform NPA}
For any ring $R$, the set of comonoform right ideals in $R$ is a closed noetherian
point annihilator set.
\end{proposition}

\begin{proof}
Because a nonzero submodule of a monoform module is again monoform,
Remark~\ref{point annihilator set restated} shows that it is enough to
check that any nonzero noetherian module $M_R$ has a monoform submodule.
This has already been noted, for example, in~\cite[4.6.5]{McConnellRobson}.
We include a separate proof for the sake of completeness.

Let $L_R\subseteq M$ be maximal with respect to the property that there
exists a nonzero cyclic submodule $N\subseteq M/L$ that can be embedded
in $M$. It is readily verified that $N$ is monoform, and writing $N\cong R/I$
for some comonoform right ideal $I$, the fact that $N$ embeds in $I$ shows
that $I$ is a point annihilator of $M$.
%
%
\end{proof}

Our most ``refined'' instance of a noetherian point annihilator set for a general
noncommutative ring is connected to the concept of (Gabriel-Rentschler) Krull
dimension. We review the relevant definitions here, and we refer the reader to the
monograph~\cite{GordonRobson} or the textbooks~\cite[Ch.~15]{GoodearlWarfield}
or~\cite[Ch.~6]{McConnellRobson} for further details. Define by induction
classes $\mathcal{K}_{\alpha}$ of right $R$-modules for each ordinal $\alpha$
(for convenience, we consider $-1$ to be an ordinal number) as follows. Set
$\mathcal{K}_{-1}$ to be the class consisting of the zero module. Then for an
ordinal $\alpha\geq 0$ such that $\mathcal{K}_{\beta}$ has been defined for
all ordinals $\beta <\alpha$, define $\mathcal{K}_{\alpha}$ to be the class
of all modules $M_R$ such that, for every descending chain
\[
M_0\supseteq M_1\supseteq M_2\supseteq\cdots
\]
of submodules of $M$ indexed by natural numbers, one has
$M_i/M_{i+1}\in\bigcup_{\beta <\alpha} \mathcal{K}_{\beta}$ for almost all
indices $i$. Now if a module $M_R$ belongs to some $\mathcal{K}_{\beta}$,
its Krull dimension, denoted $\Kdim (M)$, is defined to be the least ordinal $\alpha$ such that
$M\in\mathcal{K}_{\alpha}$. Otherwise we say that the Krull dimension of
$M$ does not exist.

From the definitions it is easy to see that the right $R$-modules of Krull
dimension 0 are precisely the (nonzero) artinian modules.  Also, a module
$M_R$ has Krull dimension 1 iff it is not artinian and in every descending
chain of submodules of $M$, almost all filtration factors are artinian.

One of the more useful features of the Krull dimension function is
that it is an \emph{exact} dimension function, in the sense that,
given an exact sequence $0 \to L \to M \to N \to 0$ of right $R$-modules,
one has
\[
\Kdim (M) = \sup ( \Kdim (L) , \Kdim (N) )
\]
where either side of the equation exists iff the other side exists.
See~\cite[Lem.~15.1]{GoodearlWarfield} or~\cite[Lem.~6.2.4]{McConnellRobson}
for details. 

The Krull dimension can also be used as a dimension measure for rings.
We define the \emph{right Krull dimension} of a ring $R$ to be
$\rKdim(R)=\Kdim(R_R)$. The left Krull dimension of $R$ is defined
similarly.

Now a module $M_R$ is said to be \emph{$\alpha$-critical} ($\alpha\geq 0$ an ordinal)
if $\Kdim (M)=\alpha$ but $\Kdim (M/N)<\alpha$ for all $0\neq N_R\subseteq M$,
and we say that $M_R$ is \emph{critical} if it is $\alpha$-critical for some ordinal
$\alpha$. With this notion in place, we define a right ideal $I_R\subseteq R$ to be
\emph{$\alpha$-cocritical} if the module $R/I$ is $\alpha$-critical, and we say
that $I$ is \emph{cocritical} if it is $\alpha$-cocritical for some ordinal
$\alpha$. Notice immediately that a 0-critical module is the same as a simple module,
and the 0-cocritical right ideals are precisely the maximal right ideals.

Cocritical right ideals were already studied by A.\,W.~Goldie in~\cite{Goldie},
though they are referred to there as ``critical'' right ideals. (The reader
should take care not to confuse this terminology with the phrase ``critical
right ideal'' used in a different sense elsewhere in the literature, as mentioned
in~\cite[\S6]{Reyes}.)

\begin{remarks}
\label{cocritical is comonoform}
The first two remarks below are known; for example,
see~\cite[\S 6.2]{McConnellRobson}. 

\begin{itemize}
\item[\normalfont{(1)}]
\emph{Every nonzero submodule $N$ of a critical module $M$ is also
critical and has $\Kdim (N)=\Kdim (M)$}. Suppose that $M$ is $\alpha$-critical.
If $\Kdim(N)<\alpha$, then because $\Kdim(M/N)<\alpha$, the exactness
of Krull dimension would imply the contradiction  $\Kdim(M)<\alpha$.
Hence $\Kdim(N)=\alpha$. Also, for any nonzero submodule
$N_{0}\subseteq N$ we have $\Kdim(N/N_{0})\leq\Kdim(M/N_{0})<\alpha$,
proving that $N$ is $\alpha$-critical.

\item[\normalfont{(2)}]
\emph{A critical module is always monoform}. Suppose that $M$ is
$\alpha$-critical and fix a nonzero homomorphism $f \colon C \to M$
where $C_{R} \subseteq M$. Because $C$ and $\im f$ are both nonzero
submodules of $M$, they are also $\alpha$-critical by~(1). Then
$\Kdim(C) = \Kdim(\im f) = \Kdim(C/\ker f)$, so we must have $\ker f=0$.
Thus $M$ is indeed monoform.

\item[\normalfont{(3)}]
\emph{Any cocritical right ideal is comonoform and, in particular, is
completely prime}. This follows immediately from the preceding remark
and the fact~\cite[Prop.~6.3]{Reyes} that any
comonoform right ideal is completely prime.
\end{itemize}
\end{remarks}

It is possible to characterize the (two-sided) ideals that are cocritical
as right ideals.

\begin{proposition}
\label{cocritical ideals}
For any ring $R$ and any ideal $P\lhd R$, the following are equivalent:
\begin{itemize}
\item[\normalfont (1)] $P$ is cocritical as a right ideal;
\item[\normalfont (2)] $R/P$ is a (right Ore) domain with right Krull dimension.
\end{itemize}
\end{proposition}

\begin{proof}
Because every cocritical right ideal is comonoform, (1)$\implies$(2) follows from
the fact~\cite[Prop.~6.5]{Reyes} that a two-sided ideal is comonoform as a
right ideal iff its factor ring is a right Ore domain. Because a semiprime ring with
right Krull dimension is right Goldie, the two conditions in~(2) are equivalent.

(2)$\implies$(1): It can be shown that, for every module $M_R$ whose Krull
dimension exists and for every injective endomorphism $f:M\to M$, one has
$\Kdim (M) > \Kdim (M/f(M))$ (see~\cite[Lem.~15.6]{GoodearlWarfield}). Applying
this to $M=S:=R/P$, we see that $\Kdim (S) > \Kdim (S/xS)$ for all nonzero
$x\in S$ (this is also proved in~\cite[Lem.~6.3.9]{McConnellRobson}). Thus
$S_S$, and consequently $S_R$, are critical modules.
\end{proof}

\begin{example}
\label{comonoform not cocritical}
The last proposition is useful for constructing an ideal of a ring that is
(right and left) comonoform but not (right or left) cocritical. If $R$
is a commutative domain that does not have Krull dimension, then the zero ideal
of $R$ is prime and thus is comonoform by~\cite[Prop.~6.5]{Reyes}.
But because $R$ does not have Krull dimension, the zero ideal cannot be cocritical
by the previous result. For an explicit example, one can take $R=k[x_1,x_2,\dots ]$
for some commutative domain $k$. It is shown in~\cite[Ex.~10.1]{GordonRobson} 
that such a ring does not have Krull dimension, using the fact that a polynomial
ring $R[x]$ has right Krull dimension iff the ground ring $R$ is right noetherian.
\end{example}

The reason for our interest in the set of cocritical right ideals is that it
is an important example of a noetherian point annihilator set in a general
ring.

\begin{proposition}
\label{cocritical NPA}
For any ring $R$, the set of all cocritical right ideals is a closed point
annihilator set for the class of right $R$-modules whose Krull dimension
exists. In particular, this set is a closed noetherian point annihilator set
for $R$.
\end{proposition}

\begin{proof}
Because any nonzero module with Krull dimension has a critical submodule
(see~\cite[Lem.~15.8]{GoodearlWarfield} or~\cite[Lem.~6.2.10]{McConnellRobson}),
Remark~\ref{point annihilator set restated} shows that the set of cocritical
right ideals of $R$ is a point annihilator set for the class of right
$R$-modules with Krull dimension. Because any noetherian module has Krull
dimension (see~\cite[Lem.~15.3]{GoodearlWarfield} or~\cite[Lem.~6.2.3]{McConnellRobson}),
we see by Remark~\ref{larger point annihilator sets} that this same set is
a right noetherian point annihilator set for $R$. The fact that this set
is closed under point annihilators follows from Remark~\ref{cocritical is comonoform}(1).
\end{proof}

\separate

Let us further examine the relationship between the general noetherian point annihilator
sets given in Propositions~\ref{completely prime NPA},~\ref{comonoform NPA},
and~\ref{cocritical NPA}. From~\cite[Prop.~6.3]{Reyes} and
Remark~\ref{cocritical is comonoform}(3) we see that there are always the following
containment relations (where the first three sets are noetherian point annihilator sets
but the last one is not, in general):
\begin{equation}
\label{inclusion of NPA sets}
\left\{
\begin{array}{c}
\text{completely}\\
\text{prime}\\
\text{right ideals}
\end{array}
\right\}\supseteq\left\{
\begin{array}{c}
\text{comonoform}\\
\text{right ideals}
\end{array}
\right\}\supseteq\left\{
\begin{array}{c}
\text{cocritical}\\
\text{right ideals}
\end{array}
\right\}\supseteq\left\{
\begin{array}{c}
\text{maximal}\\
\text{right ideals}
\end{array}
\right\}.
\end{equation}
Notice that in a commutative ring $R$ the first two sets are equal to $\Spec(R)$
by~\cite[Cor.~2.3 \& Cor.~6.7]{Reyes},
and when $R$ is commutative and has Gabriel-Rentschler Krull dimension (e.g., when $R$ is
noetherian) the third set is also equal to $\Spec(R)$ by Proposition~\ref{cocritical ideals}.
The latter fact provides many examples where the last containment is strict: in any
commutative ring $R$ with Krull dimension $>0$ there exists a nonmaximal prime ideal,
which must be cocritical.
It was shown in~\cite[Lem.~6.4~ff.]{Reyes} and Example~\ref{comonoform not cocritical}
that the first two inclusions can each be strict. However, the latter example was necessarily
non-noetherian. Below we give an example of a noncommutative \emph{artinian}
(hence noetherian) ring over which both containments are strict.
This example makes use of the following characterization of semi-artinian monoform
modules. The proof is straightforward and therefore is omitted. The socle of a
module $M_R$ is denoted by $\soc(M)$.

\begin{lemma}
\label{semi-artinian monoform}
Let $M_R$ be a semi-artinian $R$-module. Then the following are equivalent:
\begin{itemize}
\item[\normalfont (1)] $M$ is monoform;
\item[\normalfont (2)] For any nonzero submodule $K_R\subseteq M$, $\soc(M)$ and
$\soc(M/K)$ do not have isomorphic nonzero submodules;
\item[\normalfont (3)] $\soc(M)$ is simple and does not embed into any proper
factor module of $M$.
\end{itemize}
\end{lemma}



\begin{example}
Let $k$ be a division ring, and let $R$ be the ring of all $3\times 3$ matrices
over $k$ of the form
\begin{equation}
\label{triangular}
\begin{pmatrix}
a & b & c\\
0 & d & e\\
0 & 0 & d
\end{pmatrix}
.
\end{equation}
One can easily verify (for example, by passing to the factor $R/\rad(R)$ of $R$
by its Jacobson radical)
that $R$ has two simple right modules up to isomorphism. We may view these
modules as $S_1=k$ with right $R$-action given by right multiplication by $a$
in~\eqref{triangular} and $S_2=k$ with right action given by right
multiplication by $d$ in~\eqref{triangular}. Consider the right ideals 
\[
P_0 := \left\{
\begin{pmatrix}
0 & 0 & 0\\
0 & d & e\\
0 & 0 & d
\end{pmatrix}
\right\} \subseteq P_1 := \left\{
\begin{pmatrix}
0 & 0 & c\\
0 & d & e\\
0 & 0 & d
\end{pmatrix}
\right\} \subseteq P_2 := \left\{
\begin{pmatrix}
0 & b & c\\
0 & d & e\\
0 & 0 & d
\end{pmatrix}
\right\}.
\]
Then the cyclic module $V:=R/P_0$ is isomorphic to the space $(k\ k\ k)_R$
of row vectors with the natural right $R$-action. Notice that $V_i:=P_i/P_0$
($i=1,2$) corresponds to the submodule of row vectors whose first $3-i$ entries
are zero. One can check that the only submodules of $V$ are
$0\subseteq V_1 \subseteq V_2 \subseteq V$, which implies that this is the
unique composition series of $V$. It is clear that
\[
V_1 \cong V_2/V_1 \cong S_2 \quad \text{and} \quad V/V_2 \cong S_1.
\]

We claim that $P_0$ is a completely prime right ideal that is not comonoform.
To see that it is completely prime, it suffices to show that every nonzero
endomorphism of $V=R/P_0$ is injective. Indeed, the only proper factors of
$V$ are $V/V_1$ and $V/V_2$. By an inspection of composition factors, neither
of these can embed into $V$, proving that $P_0$ is completely prime. To see that
$P_0$ is not comonoform, consider that 
\[
\soc (V) = V_1 \cong V_2/V_1 =\soc (V/V_1).
\]
By Lemma~\ref{semi-artinian monoform} we see that $R/P_0=V$ is not monoform and thus
$P_0$ is not comonoform.

We also claim that $P_1$ is comonoform but not cocritical. Notice that over any
right artinian ring, every cyclic critical module has Krull dimension 0. But a
0-critical module is necessarily simple. Thus a cocritical right ideal in a right
artinian ring must be maximal. But $P_1$ is not maximal and thus is not cocritical.
On the other hand, $R/P_1\cong V/V_1$ has unique composition series
$0\subseteq V_2/V_1 \subseteq V/V_1$. This allows us to easily verify, using
Lemma~\ref{semi-artinian monoform}, that $V/V_1\cong R/P_1$ is monoform,
proving that $P_1$ is comonoform.

This same example also demonstrates that the set of completely prime right ideals
is not always closed under point annihilators (as in Definition~\ref{closed point annihilator set definition}).
This is because the cyclic submodule $V_2\subseteq V=R/P_0$ certainly has a nonzero
noninjective endomorphism, as both of its composition factors are isomorphic.
\end{example}

An example along these lines was already used in~\cite[p.~11]{GordonRobson}
to show that a monoform module need not be critical. Notice that the completely
prime right ideal $P_0$ above is such that $R/P_0$ is uniform, even if it is
not monoform. (This means that the right ideal $P_0$ is ``meet-irreducible.'')
An example of a completely prime right ideal whose factor module is not uniform was
already given in Example~{completely prime not meet-irreducible}.

\separate

Given the containments of noetherian point annihilator sets
in~\eqref{inclusion of NPA sets}, one might question the need for the
notion of a point annihilator set. Why not simply state all theorems below
just for the family of cocritical right ideals? We already have an answer to
this question in Example~\ref{left perfect example}, which
demonstrates that every left perfect ring has a \emph{finite} right noetherian
point annihilator set. The reason we can reduce to a finite set $\setS$ in
such rings is the fact stated in Remark~\ref{point annihilator set restated}
that a noetherian right module only needs to contain a submodule \emph{isomorphic to}
$R/I$ for some $I\in\setS$. In other words, $\setS$ only needs to contain a
single representative from any given similarity class. So while a left perfect
ring $R$ may have infinitely many maximal right ideals, it has only
finitely many similarity classes of maximal right ideals. Thus we can reduce
certain problems about all right ideals of $R$ to a finite set of maximal
ideals! This will be demonstrated in Proposition~\ref{semi-artinian right artinian}
 and Corollary~\ref{left perfect PRIR}, below where we shall prove that a
left perfect ring is right noetherian (resp.\ a PRIR) iff all maximal right ideals
belonging to a (properly chosen) finite set are finitely generated (resp.\ principal). 

We have also phrased the discussion in terms of general noetherian point
annihilator sets to leave open the possibility of future applications to
classes of rings which have nicer noetherian point annihilator sets than
the whole set of cocritical right ideals, akin to the class of left perfect
rings.

\section{The Point Annihilator Set Theorem}
\label{point annihilator theorem section}

Having introduced the notion of a point annihilator set, we can now state our
fundamental result, the Point Annihilator Set Theorem~\ref{first point annihilator set theorem}.
This theorem gives conditions under which one may deduce that one family $\F_0$ of right
ideals is contained in a second family $\F$ of right ideals. We will most often use it as a
sufficient condition for concluding that all right ideals of a ring lie in a particular
right Oka family $\F$.

Certain results in commutative algebra state that when every prime ideal in a
commutative ring has a certain property, then all ideals in the ring have
that property. As mentioned in the introduction, the two motivating examples
are Cohen's Theorem~\ref{original Cohen's Theorem} and Kaplansky's
Theorem~\ref{original Kap Cohen Theorem}. In~\cite[p.~3017]{LR}, these theorems
were both recovered in the context of Oka families and the Prime Ideal
Principle. The useful tool in that context was the ``Prime Ideal Principle
Supplement''~\cite[Thm.~2.6]{LR}. 
We have already provided one noncommutative generalization of this tool in the
Theorem~\ref{CPIP supplement}, which we used to produce a noncommutative
extension of Cohen's Theorem in~\cite[Thm.~3.8]{Reyes} stating that a ring is
right noetherian iff its completely prime right ideals are finitely generated.

The CPIP Supplement states that for certain right Oka families $\F$, if the
set $\setS$ of completely prime right ideals lies in $\F$, then all right ideals
lie in $\F$. The main goal of this section is to improve upon this result 
by allowing the set $\setS$ to be \emph{any} point annihilator set. This is
achieved in Theorem~\ref{"PIP supplement"} as an application of
the Point Annihilator Set Theorem.

The Point Annihilator Set Theorem basically formalizes a general ``strategy of proof.''
For the sake of clarity, we present an informal sketch of this proof strategy before stating
the theorem.
Suppose that we want to prove that \emph{every module with the property $\mathcal{P}$
also has the property $\mathcal{Q}$}.
Assume for contradiction that there is a counterexample. Use Zorn's Lemma to pass
to a counterexample $M$ satisfying $\mathcal{P}$ that is ``critical'' with respect to not
satisfying $\mathcal{Q}$, in the sense that every proper factor module of $M$
satisfies $\mathcal{Q}$ but $M$ itself does not satisfy $\mathcal{Q}$. Argue that $M$
has a nonzero submodule $N$ that satisfies $\mathcal{Q}$. Finally, use the fact that
$N$ and $M/N$ have $\mathcal{Q}$ to deduce the contradiction that $M$ has $\mathcal{Q}$.

Our theorem applies in the specific case where one's attention is restricted to
cyclic modules. In the outline above, we may think of the properties $\mathcal{P}$ and
$\mathcal{Q}$ to be, respectively, ``$M=R/I$ where $I\in\F_0$'' and ``$M=R/I$ where
$I\in\F$.''

\begin{theorem}[The Point Annihilator Set Theorem]
\label{first point annihilator set theorem}
Let $\F$ be a right Oka family such that every nonempty chain of right ideals
in $\F'$ (with respect to inclusion) has an upper bound in $\F'$. 
\begin{itemize}
\item[\normalfont (1)] Let $\F_0$ be a semifilter of right ideals in $R$. If $\F$ is
a point annihilator set for the class of modules $\{ R/I : I_R \in \Max(\F')\cap\F_0 \}$, 
then $\F_0\subseteq\F$.
\item[\normalfont (2)] For any right ideal $J_R\subseteq R$, if $\F$ is a point annihilator
set for the class of modules $R/I$ such that $I\in\Max(\F')$ and $I\supseteq J$
(resp.\ $I\supsetneq J$), then all right ideals containing (resp.\ properly
containing) $J$ belong to $\F$.
\item[\normalfont (3)] If $\F$ is a point annihilator set for the class of modules $\{R/I : I_R\in\Max(\F')\}$,
then $\F$ consists of all right ideals of $R$.
\end{itemize}
\end{theorem}

\begin{proof}
Suppose that the hypotheses of~(1) hold, and assume for contradiction
that there exists $I_0\in\F_0\setminus\F$. The assumptions on~$\F'$ allow us
to apply Zorn's Lemma to find $I\in\Max(\F')$ with $I\supseteq I_0$. Then
$I\in\F_0$ because $\F_0$ is a semifilter. The point annihilator hypothesis implies
that there is a nonzero element $a+I\in R/I$ such that $a^{-1}I=\ann(a+I)\in\F$.
On the other hand, $a+I\neq 0+I$ implies that $I+aR\supsetneq I$.
By maximality of $I$, this means that $I+aR\in\F$. Because $\F$ is a right Oka
family, we arrive at the contradiction $I\in\F$.

Parts~(2) and~(3) follows from~(1) by taking $\F_0$ to be, respectively, the
set of all right ideals of $R$ (properly) containing $J$ or the set of all right ideals of $R$.
\end{proof}

Notice that part~(1) above remains true if we weaken the condition on chains in
$\F'$ to the following: every nonempty chain in $\F'\cap\F_0$ has an upper bound in
$\F'$. The latter condition holds if every $I\in\F_0$ is such that $R/I$ is a noetherian
module, or more generally if $\F_0$ satisfies the ascending chain condition (as a partially
ordered set with respect to inclusion). However, we shall not make use this observation
in the present work.

The following is an illustration of how Theorem~\ref{first point annihilator set theorem}
can be applied in practice. It is well-known that every finitely
generated artinian module over a commutative ring has finite length.
However, there exist finitely generated (even cyclic) artinian right
modules over noncommutative rings that do not have finite length; for instance,
see~\cite[Ex.~4.28]{ExercisesClassical}. Here we provide a sufficient condition
for all finitely generated artinian right modules over a ring to have finite
length. 

\begin{proposition}
If all maximal right ideals of a ring $R$ are finitely generated, then
every finitely generated artinian right $R$-module has finite length.
\end{proposition}

\begin{proof}
It suffices to show that every cyclic artinian right $R$-module has finite
length.
Let $\F_0$ be the semifilter of right ideals $I_R$ such that $R/I$ is right artinian,
and let $\F$ be the right Oka family of right ideals $I$ such that $R/I$ has finite
length. Our goal is then to show that $\F_0\subseteq\F$. Because every nonzero
cyclic artinian module has a simple submodule, we see that $\F$ is a point annihilator
set for the class $\{R/I : I\in\F_0\}\supseteq\{R/I : I\in\Max(\F')\cap\F_0\}$.
To apply Theorem~\ref{first point annihilator set theorem}(1) we will
show that every nonempty chain in $\F'$ has an upper bound in $\F'$. For this,
it is enough to check that $\F$ consists of finitely generated right ideals.
The hypothesis implies that all simple right $R$-modules are finitely presented.
If $I\in\F$ then $R/I$, being a repeated extension of finitely many simple
modules, is finitely presented. It follows that $I$ is finitely generated. (The
details of the argument that $\F$ consists of f.g.\ right ideals are in~\cite[Cor~4.9]{Reyes}.)
\end{proof}

In light of the result above, it would be interesting to find a characterization of
the rings $R$ over which every finitely generated artinian right $R$-module has
finite length.  How would such a characterization unite both commutative rings
and the rings in which every maximal right ideal is finitely generated?

\separate

For our purposes, it will often best to use a variant of the theorem above.
This variant keeps with the theme of Cohen's and Kaplansky's results
(Theorems~\ref{original Cohen's Theorem}--\ref{original Kap Cohen Theorem}) of
``testing'' a property on special sets of right ideals.

\begin{theorem}
\label{"PIP supplement"}
Let $\F$ be a right Oka family such that every nonempty chain of right ideals
in $\F^{\prime}$ (with respect to inclusion) has an upper bound in $\F^{\prime}$.
Let $\setS$ be a set of right ideals that is a point annihilator set for the
class of modules $\{ R/I:I_R\in \Max (\F')\}$.
\begin{itemize}
\item[\normalfont (1)] Let $\F_0$ be a divisible semifilter of right ideals in
$R$. If $\F_0\cap\setS\subseteq\F$, then $\F_0\subseteq\F$.
\item[\normalfont (2)] For any ideal $J\lhd R$, if all right ideals in $\setS$
that contain $J$ belong to $\F$, then every right ideal containing $J$ belongs
to $\F$.
\item[\normalfont (3)]If $\mathcal{S}\subseteq \F$, then all right ideals of $R$
belong to $\F$.
\end{itemize}
\end{theorem}

\begin{proof}
As in the previous result, parts~(2) and~(3) are special cases of part~(1). To
prove~(1), Theorem~\ref{first point annihilator set theorem} implies that it is
enough to show that $\F$ is a point annihilator set for the class of modules
$\{R/I:I_R\in\Max(\F')\cap\F_0\}$. Fixing such $R/I$, the hypothesis of part~(1)
ensures that $R/I$ has a point annihilator in $\setS$, say $A=\ann(x+I)\in\setS$
for some $x+I\in R/I\setminus\{0+I\}$. Because $I\in\F_0$ and $\F_0$ is divisible,
the fact that $A=x^{-1}I$ implies that $A\in\F_0$. Thus $A\in\setS\cap\F_0\subseteq\F$,
providing a point annihilator of $R/I$ that lies in $\F$.
\end{proof}

We also record a version of Theorem~\ref{"PIP supplement"} adapted
especially for families of finitely generated right ideals. Because of its easier
formulation, it will allow for simpler proofs as we provide applications of
Theorem~\ref{"PIP supplement"}.

\begin{corollary}
\label{f.g. "PIP supplement"}
Let $\F$ be a right Oka family in a ring $R$ that consists of finitely generated right
ideals. 	Let $\setS$ be a noetherian point annihilator set for $R$. Then the
following are equivalent:
\begin{itemize}
\item[\normalfont (1)] $\F$ consists of all right ideals of $R$;
\item[\normalfont (2)] $\F$ is a noetherian point annihilator set;
\item[\normalfont (3)] $\setS\subseteq\F$.
\end{itemize}
\end{corollary}

\begin{proof}
Given any $I\in\Max(\F')$, any nonzero submodule of $R/I$ is the image of
a right ideal properly containing $I$, which must be finitely generated;
thus $R/I$ is a noetherian right $R$-module. Stated another way, the class
$\{R/I : I\in\Max(\F')\}$ consists of noetherian modules. Thus (1)$\iff$(2)
 follows from Theorem~\ref{first point annihilator set theorem}(3) and
(1)$\iff$(3) follows from Theorem~\ref{"PIP supplement"}(3).
\end{proof}

\separate

As our first application of the simplified corollary above, we will
finally present our noncommutative generalization of Cohen's
Theorem~\ref{original Cohen's Theorem}, improving upon~\cite[Thm.~3.8]{Reyes}.

\begin{theorem}[A noncommutative Cohen's Theorem]
\label{Cohen's Theorem}
Let $R$ be a ring with a right noetherian point annihilator set $\setS$.
The following are equivalent:
\begin{itemize}
\item[\normalfont (1)] $R$ is right noetherian;
\item[\normalfont (2)] Every right ideal in $\setS$ is finitely generated;
\item[\normalfont (3)] Every nonzero noetherian right $R$-module has a finitely generated point annihilator;
\item[\normalfont (4)] Every nonzero noetherian right $R$-module has a nonzero cyclic finitely presented
submodule.
\end{itemize}
In particular, $R$ is right noetherian iff every cocritical right ideal
is finitely generated.
\end{theorem}

\begin{proof}
The family of finitely generated right ideals is a right Oka family by~\cite[Prop.~3.7]{Reyes}.
The equivalence of~(1), (2), and~(3) thus follows directly from
Corollary~\ref{f.g. "PIP supplement"}. Also, (3)$\iff$(4) comes from the
observation that a right ideal $I$ is a point annihilator of a module $M_R$ iff
there is an injective module homomorphism $R/I\hookrightarrow M$, as well as the fact
that $R/I$ is a finitely presented module iff $I$ is a finitely generated right
ideal~\cite[(4.26)(b)]{Lectures}. The last statement follows from Proposition~\ref{cocritical NPA}.
\end{proof}

In particular, if we take the set $\setS$ above to be the completely prime
right ideals of $R$, we recover~\cite[Thm.~3.8]{Reyes}. Our version of Cohen's
Theorem will be compared and contrasted with earlier such generalizations
in~\S\ref{previous generalizations section}.

The result above suggests that one might wish to drop the word ``cyclic'' in
characterization~(4). This is indeed possible. We present this as a separate result
since it does not take advantage of the ``formalized proof method'' given in
Theorem~\ref{first point annihilator set theorem}. 
However, this result does follow the informal ``strategy of proof'' outlined at the
beginning of this section.

\begin{proposition}
For a ring $R$, the following are equivalent:
\begin{itemize}
\item[\normalfont (1)] $R$ is right noetherian;
\item[\normalfont (5)] Every nonzero noetherian right $R$-module has a nonzero finitely
presented submodule.
\end{itemize}
\end{proposition}

\begin{proof}
Using the numbering from Theorem~\ref{Cohen's Theorem}, we have
(1)$\implies$(4)$\implies$(5). Suppose that~(5) holds, and assume for contradiction
that there exists a right ideal of $R$ that is not finitely generated. Using Zorn's Lemma,
pass to $I_R\subseteq R$ that is maximal with respect to not being finitely generated.
Then because every right ideal properly containing $I$ is f.g., the module $R/I$ is 
noetherian. By hypothesis, there is a finitely presented submodule $0\neq J/I\subseteq R/I$.
Then $J\supsetneq I$ implies that $J$ is finitely generated, so that $R/J$ is finitely
presented. Because $R/I$ is an extension of the two finitely presented modules $J/I$
and $R/J$, $R/I$ is finitely presented~\cite[Ex.~4.8(2)]{ExercisesModules}. But if $R/I$
is finitely presented then $I_R$ is finitely generated~\cite[(4.26)(b)]{Lectures}. This
is a contradiction.
\end{proof}

\separate

The Akizuki-Cohen Theorem of commutative algebra (cf.~\cite[pp.~27--28]{Cohen})
states that a commutative ring $R$ is artinian iff it is noetherian and every
prime ideal is maximal.
Recall that a module $M_R$ is \emph{finitely cogenerated} if any family of
submodules of $M$ whose intersection is zero has a finite subfamily whose
intersection is zero. In~\cite[(5.17)]{LR} consideration of the class of finitely
cogenerated right modules led to the following ``artinian version'' of
Cohen's theorem: a commutative ring $R$ is artinian iff for all $P\in\Spec(R)$,
$P$ is finitely generated and $R/P$ is finitely cogenerated. Here we generalize
both of these results to the noncommutative setting.

\begin{proposition}
For a ring $R$ with right noetherian point annihilator set $\setS$, the
following are equivalent:
\begin{itemize}
\item[\normalfont (1)] $R$ is right artinian;
\item[\normalfont (2)] $R$ is right noetherian and for all $P\in\setS$, $(R/P)_R$
has finite length;
\item[\normalfont (3)] For all $P\in\setS$, $P_R$ is finitely generated and $(R/P)_R$
has finite length;
\item[\normalfont (4)] For all $P\in\setS$, $P_R$ is finitely generated and $(R/P)_R$
is finitely cogenerated;
\item[\normalfont (5)] $R$ is right noetherian and every cocritical right ideal of $R$
is maximal;
\item[\normalfont (6)] Every cocritical right ideal of $R$ is finitely generated and maximal.
\end{itemize}
\end{proposition}

\begin{proof}
(1)$\iff$(2)$\iff$(3): It is well-known that $R$ is right artinian iff
$R_R$ has finite length. This equivalence then follows from
Corollary~\ref{f.g. "PIP supplement"}, Theorem~\ref{Cohen's Theorem}, and
the fact that $\F:=\{I_R\subseteq R : R/I_R\text{ has finite length}\}$ is a right Oka
family (see~\cite[Ex.~5.18(4)]{Reyes}).

(1)$\iff$(4): It is known that a module $M_R$ is artinian iff
every quotient of $M$ is finitely cogenerated (see~\cite[Ex.~19.0]{ExercisesModules}).
Because $\F:=\{I_R\subseteq R : R/I_R\text{ is finitely cogenerated}\}$ is a
right Oka family (see~\cite[Ex.~5.18(1B)]{Reyes}), (1)$\iff$(4) follows from
Corollary~\ref{f.g. "PIP supplement"}.

We get (1)$\iff$(5)$\iff$(6) by applying the equivalence of~(1),~(2),
and~(3) to the case where $\setS$ is the set of cocritical right ideals
of $R$, noting that every artinian critical module is necessarily simple.
\end{proof}

Of course, the fact that a right noetherian ring is right artinian iff all
of its cocritical right ideals are maximal follows from a direct argument
involving Krull dimensions of modules. Indeed, given a right noetherian ring
$R$ with right Krull dimension $\alpha$, choose a right ideal $I\subseteq R$
maximal with respect to $\Kdim(R/I)=\alpha$. Then for any right ideal
$J\supseteq I$, $\Kdim(R/J)<\alpha=\Kdim(R/I)$; hence $I$ is cocritical. So
\[
\rKdim(R) = \sup \{ \Kdim(R/I) : I_R \subseteq R \text{ is cocritical} \}.
\]
The result now follows once we recall that the 0-critical modules are precisely
the simple modules.

We also mention another noncommutative generalization of the Akizuki-Cohen
Theorem due to A.~Kert\'esz, which states that a ring $R$ is right artinian iff it is right
noetherian and for every prime ideal $P \lhd R$, $R/P$ is right artinian~\cite{Kertesz}.
(We thank the referee for bringing this reference to our attention.)

Another application of Theorem~\ref{Cohen's Theorem} tells us when a right
semi-artinian ring, especially a \emph{left} artinian ring, is right
artinian. (The definition of a right semi-artinian ring was recalled in
Example~\ref{semi-artinian example}.)

\begin{proposition}
\label{semi-artinian right artinian}
\textnormal{(1)} A right semi-artinian ring $R$ is right artinian iff every maximal
right ideal of $R$ is finitely generated.

\textnormal{(2)} Let $R$ be a left perfect ring (e.g.\ a semiprimary ring, such as a
left artinian ring) and let $\m_1,\dots,\m_n$ be maximal right ideals
such that $R/\m_i$ exhaust all isomorphism classes of simple right modules.
Then $R$ is right artinian iff all of the $\m_i$ are finitely generated.
\end{proposition}

\begin{proof}
It is easy to check that a right semi-artinian ring $R$ is right artinian
iff it is right noetherian. The proposition then follows from Theorem~\ref{Cohen's Theorem}
and Examples~\ref{semi-artinian example}--\ref{left perfect example}.
\end{proof}

A result of B.~Osofsky~\cite[Lem.~11]{Osofsky} states that a left or right perfect ring
$R$ with Jacobson radical $J$ is right artinian iff $J/J^2$ is finitely generated as a
right $R$-module. This applies, in particular, to left artinian rings.
D.\,V.~Huynh characterized which (possibly nonunital) left artinian rings are right artinian
in~\cite[Thm.~1]{Huynh}. In the unital case, his characterization recovers Osofsky's
result above for the special class of left artinian rings.
We can use our previous result to recover a weaker version of Osofsky's theorem
that implies Huynh's result for unital left artinian rings.

\begin{corollary}
\label{Huynh's theorem}
Let $R$ be a ring with $J:=\rad(R)$. The the following are
equivalent:
\begin{itemize}
\item[\normalfont (1)] $R$ is right artinian;
\item[\normalfont (2)] $R$ is left perfect and $J$ is a finitely generated
right ideal;
\item[\normalfont (3)] $R$ is perfect and $J/J^2$ is a finitely generated
right $R$-module.
\end{itemize}
In particular, if $R$ is semiprimary (for instance, if it is left artinian),
then $R$ is right artinian iff $J/J^2$ is finitely generated on the right.
\end{corollary}

\begin{proof}
Because any right artinian ring is both perfect and right noetherian, we
have (1)$\implies$(3). For (3)$\implies$(2), suppose that $R$ is perfect
and that $J/J^2$ is right finitely generated. Then for some finitely
generated submodule $M_R\subseteq J_R$, $J=M+J^2$. Since $R$ is right
perfect, $J$ is right T-nilpotent. Then by ``Nakayama's Lemma'' for right
T-nilpotent ideals (see~\cite[(23.16)]{FC}) implies that $J_R=M_R$ is
finitely generated.

Finally we show (2)$\implies$(1). Suppose that $R$ is left perfect and
that $J_R$ is finitely generated. For any maximal right ideal $\m$ of $R$,
we have $J\subseteq\m$. Now $\m/J$ is a right ideal of the semisimple
ring $R/J$ and is therefore finitely generated. Because $J_R$ is also finitely
generated, we see that $\m_R$ itself is finitely generated. Since this
is true for all maximal right ideals of $R$, Proposition~\ref{semi-artinian right artinian}(2)
implies that $R$ is right artinian.
\end{proof}

\separate

Next we give a condition for every finitely generated right module over a ring $R$
to have a finite free resolution (FFR). Notice that such a ring is necessarily right
noetherian. Indeed, any module with an FFR is necessarily finitely
presented. Thus if every f.g.\ right $R$-module has an FFR, then for every right
ideal $I\subseteq R$ the module $R/I$ must have an FFR and therefore must be
finitely presented. It follows (from Schanuel's Lemma~\cite[(5.1)]{Lectures}) that
$I_R$ is finitely generated, and $R$ is right noetherian.

\begin{proposition}
Let $\setS$ be a right noetherian point annihilator set for a ring $R$ (e.g.\ the set of
cocritical right ideals). Then the following are equivalent.
\begin{itemize}
\item[\normalfont (1)] Every finitely generated right $R$-module has a finite free resolution;
\item[\normalfont (2)] For all $P\in\setS$, $R/P$ has a finite free resolution;
\item[\normalfont (3)] Every right ideal in $\setS$ has a finite free resolution.
\end{itemize}
\end{proposition}

\begin{proof}
(1)$\implies$(3): As mentioned before the proposition, if every f.g.\ right
$R$-module has a finite free resolution then $R$ is right noetherian. So
every right ideal $I_R\subseteq R$ is finitely generated and therefore has
a finite free resolution.

Next, (3)$\implies$(2) follows from the easy fact that, given $I_R\subseteq R$,
if $I$ has a finite free resolution then so does $R/I$.
For (2)$\implies$(1), let $\F$ be the family of right ideals $I$ such that
$R/I$ has a finite free resolution and assume that $\setS\subseteq\F$. This
is a right Oka family according to~\cite[Ex.~5.12(5)]{Reyes}.
Moreover, if $I\in\F$ then $R/I$ is finitely presented. As noted earlier, this implies
that $I_R$ must be finitely generated~\cite[(4.26)(b)]{Lectures}. It follows
from Corollary~\ref{f.g. "PIP supplement"} that every right ideal of $R$ lies in
$\F$. Because any finitely generated right $R$-module is an extension of cyclic
modules and because the property of having an FFR is preserved by extensions,
we conclude that (1) holds.
\end{proof}

\section{Families of principal right ideals}
\label{principal right ideal section}

We will use $\Fpr(R)$ to denote the family of principal right ideals of
a ring $R$. If the ring $R$ is understood from the context, we may simply
use $\Fpr$ to denote this family. 

A theorem of Kaplansky~\cite[Thm.~12.3 \& Footnote~8]{KapDivisors} states
that a commutative ring is a principal ideal ring iff its prime ideals are all
principal. In~\cite[(3.17)]{LR} this theorem was recovered via the ``PIP supplement.''
It is therefore reasonable to hope that the methods presented here will lead
to a generalization of this result. Specifically, we would like to know whether
a ring $R$ is a principal right ideal ring (PRIR) if, say, every cocritical
right ideal is principal. It turns out that this is in fact true, but the path to
proving the result is not as straightforward as one might imagine. The obvious
starting point is to ask whether the family $\Fpr$ of principal right ideals in
an arbitrary ring $R$ is a right Oka family. Suppose that $R$ is a ring such that
$\Fpr$ is a right Oka family. Then Corollary~\ref{f.g. "PIP supplement"} readily applies
to $\Fpr$. However, it is not immediately clear whether or not $\Fpr(R)$ is
necessarily right Oka for every ring $R$. The following proposition provides
some guidance in this matter.

\begin{proposition}
\label{when principals are Oka}
Let $S\subseteq R$ be a multiplicative set. Then $\F:=\{sR:s\in S\}$ is a
right Oka family iff it is closed under similarity. In particular, for any
ring $R$, the family $\Fpr$ of principal right ideals is a right Oka family
iff it is closed under similarity.
\end{proposition}

\begin{proof}
Any right Oka family is closed under similarity by Proposition~ref{Oka is similarity closed}.
Conversely, assume that the family $\F$ in question is closed under similarity.
Suppose that $I+aR,\ a^{-1}I\in \F$, and write $I+aR=sR$ for some $s\in S$.
In the short exact sequence of right $R$-modules
\[
0 \to \frac{I+aR}{I} \to \frac{R}{I} \to \frac{R}{I+aR} \to 0,
\]
observe that $R/(a^{-1}I) \cong (I+aR)/I = sR/I \cong R/(s^{-1}I)$. Because
$\F$ is closed under similarity and $a^{-1}I\in \F$, we must also have
$s^{-1}I\in \F$. Fix $t\in S$ such that $s^{-1}I=tR$. Then because
$I\subseteq I+aR=sR$ we have $I=s(s^{-1}I) =stR$, and $st\in S$ implies
that $I\in\F$.
\end{proof}

In particular, we have the following ``first approximation'' to our desired
theorem.

\begin{corollary}
\label{weak Kap}
Let $\setS$ be a right noetherian point annihilator set for $R$. The following
are equivalent:
\begin{itemize}
\item[\normalfont (1)] $R$ is a principal right ideal ring;
\item[\normalfont (2)] $\Fpr$ is closed under similarity and every right ideal in $\setS$ is principal;
\item[\normalfont (3)] $\Fpr$ is closed under similarity and is a right noetherian point annihilator set.
\end{itemize}
\end{corollary}

\begin{proof}
If $R$ is a PRIR, then $\Fpr$ is equal to the family of \emph{all} right ideals in
$R$ and therefore is closed under similarity.
Also, by Proposition~\ref{when principals are Oka}, if $\Fpr$ is closed under similarity
then it is a right Oka family.
These observations along with Corollary~\ref{f.g. "PIP supplement"} establish the
equivalence of~(1)--(3).
\end{proof}

This provides some motivation to explore for which rings the family $\Fpr$
is closed under similarity (and consequently is a right Oka family). Recall that
a ring $R$ is called \emph{right duo} if every right ideal of $R$ is a two-sided
ideal. It is easy to see that in any right duo ring, and particularly in any
commutative ring, every family of right ideals is closed under similarity.
This is because in such a ring $R$, any right ideal $I$
is necessarily a two-sided ideal, so that $I=\ann(R/I)$ can be recovered from
the isomorphism class of $R/I$. Thus Proposition~\ref{when principals are Oka} applies
to show that $\Fpr$ is a right Oka family whenever $R$ is a right duo ring,
such as a commutative ring. For commutative rings $R$, the fact that $\Fpr$
is an Oka family was already noted in~\cite[(3.17)]{LR}.

Another collection of rings in which $\Fpr$ is closed under similarity is
the class of local rings. 
To show that this is the case, we use the fact~\cite[Prop.~4.6]{Reyes}
that a family $\F$ of right ideals of a ring $R$ is closed under similarity iff,
for every element $a$ and right ideal $I$ of $R$, $I+aR = R$ and
$a^{-1}I\in\F$ imply $I\in\F$.
Suppose that $R$ is local, and that $I_R\subseteq R$ and $a\in R$ are such
that $I+aR=R$ and $a^{-1}I=xR$ is principal. We want to conclude that
$I$ is principal.
Write $1=i_{0}+ar$ for some $i_0\in I$ and $r\in R$. Let $U(R)$ denote
the group of units of $R$. If $i_0\in U(R)$, then $I=R$ is principal. Else
$i_0\notin U(R)$ implies that $1-i_0=ar\in U(R)$ and hence $a\in U(R)$
($R$ local implies that right invertible elements are invertible). But then
$a^{-1}I=a^{-1}\cdot I$, so that $I=a(a^{-1}I)=axR$ is principal as desired.

\begin{remark}
\label{almost principal}
In any ring $R$, let $I_R, J_R \subseteq R$ be right ideals such that $J=xR$ is
principal and $R/I \cong R/J$. Then $I$ is generated by at most two elements.
To see this, apply Schanuel's Lemma (for instance, see~\cite[(5.1)]{Lectures})
to the exact sequences
\begin{align*}
0\rightarrow I\rightarrow R &\rightarrow R/I\rightarrow 0 \quad \text{and} \\
0\rightarrow J\rightarrow R &\rightarrow R/J\rightarrow 0
\end{align*}
to get $R\oplus I\cong R\oplus J$. The latter module is generated by at most
two elements. Therefore $I$, being isomorphic to a direct summand of this
module, is generated by at most two elements. Thus we see that such $I$ is 
``not too far'' from being principal. (Of course, the same argument shows that
if $J_R\subseteq R$ is generated by at most $n$ elements and if $I_R\subseteq R$
is similar to $J$, then $I$ is generated by at most $n+1$ elements.)
\end{remark}

The analysis above also provides the following useful fact: if the module
$R_R$ is cancellable in the category of (finitely generated) right $R$-modules
(or even in the category of finite direct sums of f.g.\ right ideals), then
the family $\Fpr$ is closed under similarity (and hence is a right Oka family).
Indeed, if this is the case, suppose that $R/I\cong R/J$ for right ideals $I$
and $J$ with $J$ principal. By the remark above, we have $I$ finitely generated
and $R\oplus I\cong R\oplus J$. With the assumption on $R_R$ we would have
$I_R\cong J_R$ principal, proving $\Fpr$ to be closed under similarity.
(In fact one can similarly show that, over such rings, the minimal number of
generators $\mu (I)$ of a f.g.\ right ideal $I\subseteq R$ is an invariant of
the similarity class of $I$.)

This provides another class of rings for which $\Fpr$ is a right Oka family,
as follows. Recall that a ring $R$ is said to have \emph{(right) stable range~1}
if, for $a,b\in R$, $aR+bR=R$ implies that $(a+br)R=R$ for some $r\in R$
(see~\cite[\S 1]{CrashCourse} for details). In~\cite[Thm.~2]{Evans73}
E.\,G.~Evans showed that for any ring with stable range~1, $R_R$ is cancellable
in the full module category $\M_R$. Thus for any ring $R$ with stable range~1,
$\Fpr(R)$ is a right Oka family. 
The class of rings with stable range~1 includes all semilocal rings (see~\cite[(20.9)]{FC}
or~\cite[(2.10)]{CrashCourse}), so that this generalizes the case of local rings
discussed above.

A similar argument applies in the class of 2-firs. A ring $R$ is said
to be a \emph{$2$-fir} (where ``fir'' stands for ``free ideal ring'') if
the free right $R$-module of rank 2 has invariant basis number and
every right ideal of $R$ generated by at most two elements is free.
We claim that $\Fpr(R)$ is closed under similarity if $R$ is a 2-fir.
Suppose that $I_R\subseteq R$ is similar to a principal right ideal $J$.
As before, we have $R\oplus I\cong R\oplus J$, and $I$ is generated
by at most two elements. So $I\cong R^m$, and $J\cong R^n$ where
$n\leq 1$ because $J$ is princpal. Thus $R^{m+1}\cong R^{n+1}$ with
$n+1\leq 2$, and the invariant basis number of the latter free module
implies that $m=n\leq 1$. Hence $I_R\cong R^m$ is a principal right
ideal.

There is yet another way in which $\Fpr(R)$ can be closed under similarity.
Suppose that every finitely generated right ideal of $R$ is principal; rings
satisfying this property are often called \emph{right B\`{e}zout rings}. Then
$\Fpr$ is equal to the set of all f.g.\ right ideals of $R$, which is a right Oka
family by~\cite[Prop.~3.7]{Reyes}.  A familiar class of examples
of such rings is the class of von Neumann regular rings; in such rings, every
finitely generated right ideal is a direct summand of $R_R$, and therefore is
principal.

We present a summary of the examples above.

\begin{examples}
\label{Fpr Oka examples}
In each of the following types of rings, the family $\Fpr$ is closed under
similarity and thus is a right Oka family:
\begin{itemize}
\item[\normalfont (1)] Right duo rings (including commutative rings);
\item[\normalfont (2)] Rings with stable range 1 (including semilocal rings);
\item[\normalfont (3)] 2-firs;
\item[\normalfont (4)] Right B\`{e}zout rings (including von Neumann regular rings).
\end{itemize}
\end{examples}

One collection of semilocal rings that we have already mentioned is the
class of left perfect rings. An application of Corollary~\ref{weak Kap}
in this case gives the following.

\begin{corollary}
\label{left perfect PRIR}
Let $R$ be a left perfect ring (e.g.\ a semiprimary ring, such as a one-sided
artinian ring), and let $\m_1,\dots,\m_n\subseteq R$ be maximal right ideals
such that the $R/\m_i$ represent all isomorphism classes of simple right
$R$-modules. Then $R$ is a PRIR iff all of the $\m_i$ are principal right
ideals.
\end{corollary}

\begin{proof}
By Example~\ref{Fpr Oka examples}(2), $\Fpr$ is an Oka family of right
ideals in $R$. By Example~\ref{left perfect example}, the set
$\{\m_i\}$ is a right noetherian point annihilator set. The claim then
follows from Corollary~\ref{weak Kap}.
\end{proof}

\separate

As it turns out, the family $\Fpr$ can indeed fail to be right Oka, even
in a noetherian domain! This will be shown in Example~\ref{Fpr is not Oka}
below, with the help of the following lemma.

\begin{lemma}
\label{x^-1 distributes}
Let $R$ be a ring with an element $x\in R$ that is not a left zero-divisor.
\begin{itemize}
\item[\normalfont (A)] If $J$ and $K$ are right ideals of $R$ with $J\subseteq xR$,
then
\[
x^{-1}(J+K) = x^{-1}J + x^{-1}K.
\]
\item[\normalfont (B)] For any $f\in R$, 
\[
x^{-1}(xf R + (1+xy) R) = fR + (1+yx)R.
\]
\end{itemize}
\end{lemma}

\begin{proof}
(A) The containment ``$\supseteq$'' holds without any assumptions on $x$, $J$, or $K$
because $x(x^{-1}J+x^{-1}K)=x\cdot (x^{-1}J)+x\cdot (x^{-1}K)\subseteq J+K$. To
show ``$\subseteq$'' let $f\in x^{-1}(J+K)$, so that there exist $j\in J$ and $k\in K$
such that $xf=j+k$. Because $J\subseteq xR$, there exists $j_0$ such that
$j=xj_0$; notice that $j_0\in x^{-1}J$. Then we have $k=xk_0$ for $k_0=f-j_0\in x^{-1}K$.
Now $xf=xj_0+xk_0$, and because $x$ is not a left zero-divisor we have
$f=j_0+k_0\in x^{-1}J+x^{-1}K$.

(B) Setting $J=xfR$ and $K=(1+xy)R$, one may compute that $x^{-1}J=fR$ and
$x^{-1}K=(1+yx)R$ (using the fact that $x$ is not a left zero divisor). The claim
follows directly from part~(A).
\end{proof}

\begin{example}
\label{Fpr is not Oka}
\emph{A ring in which $\Fpr$ is not a right Oka family.}
Let $k$ be a field and let $R:=A_1(k)=k\langle x,y:xy=yx+1\rangle$ be the
first Weyl algebra over $k$. Then $R$ is known to be a noetherian domain (which
is simple if $k$ has characteristic 0).
Define the right ideal 
\[
I_R := x^2 R + (1+xy) R \subseteq R,
\]
which is shown to be nonprincipal in~\cite[7.11.8]{McConnellRobson}. Because
$I+xR$ contains both $1+xy\in I$ and $xy\in xR$, we must have $1\in I+xR=R$.

Because $1+yx=xy\in xR$, Lemma~\ref{x^-1 distributes}(B) above (with $f=x$)
implies that $x^{-1}I=xR+(1+yx)R=xR$.
Therefore we have $I+xR=R$ and $x^{-1}I=xR$ both members of  $\Fpr$
with $I\notin \Fpr$ proving that $\Fpr$ is not a right Oka family. In
fact we have $R/I\cong R/xR$ where $I$ is not principal (the isomorphism
follows from Lemma~\ref{isomorphic cyclic modules}), showing explicitly that $\Fpr$
is not closed under similarity as predicted by Proposition~\ref{when principals are Oka}.
In agreement with Remark~\ref{almost principal}, $I$ is generated by two
elements.

Notice that $R/xR\cong k[y]$, where $k[y]\subseteq R$ acts by right
multiplication and $x\in R$ acts as $-\partial / \partial y$. If $k$
has characteristic 0 then this module is evidently simple, and because
$R/I\cong R/xR$ we see that $I$ is a maximal right ideal. If instead
$\characteristic(k)=p>0$, then $R/xR\cong k[y]$ is evidently not simple,
and not even artinian (the submodules $y^{np}k[y]$ form a strictly
descending chain for $n\geq 0$). But every proper factor of this
module has finite dimension over $k$ and is therefore artinian. So we
see that $R/I\cong R/xR$ is 1-critical, making $I$ a 1-cocritical right
ideal. Thus regardless of the characteristic of $k$, the nonprincipal
right ideal $I$ is cocritical.

On the other hand, when $\characteristic k=0$ the ring $\mathbb{M}_2(R)$ is
known to be a principal (right and left) ideal ring---see~\cite[7.11.7]{McConnellRobson}.
Then $\Fpr(\mathbb{M}_2(R))$ is equal to the set of \emph{all} right ideals
in $\mathbb{M}_2(R)$ and thus is a right Oka family. So we see that the
property ``$\Fpr(R)$ is a right Oka family'' is not Morita invariant.\qed
\end{example}

\separate

It would be very desirable to eliminate the condition in Corollary~\ref{weak Kap}
that $\Fpr$ is closed under similarity. It turns out that a suitable
strengthening of the hypothesis on the point annihilator set $\setS$
will in fact allow us to discard that assumption. The following
constructions will help us achieve this goal in Theorem~\ref{Kap} below.
Recall that for right ideals $I$ and $J$ of a ring $R$, we write
$I \sim J$ to mean that $I$ and $J$ are similar.

\begin{definition}
For any ring $R$, we define
\begin{align*}
\Fprsim(R) &:= \{ I_R \subseteq R : \forall J_R \subseteq R,\ I \sim J \implies J \in \Fpr \} \\
&\phantom{:}= \{ I_R \subseteq R : I \text{ is only similar to principal right ideals} \}.
\end{align*}
Alternatively, $\Fprsim$ is the largest subset of $\Fpr$ that is
closed under similarity.
\end{definition}

As with $\Fpr$, we will often write $\Fprsim$ in place of $\Fprsim(R)$
when the ring $R$ is understood from the context. We saw in
Proposition~\ref{when principals are Oka} that certain families of principal
right ideals are right Oka precisely when they are closed under similarity.
But $\Fprsim$ is the \emph{largest} family of principal right ideas that
is closed under similarity. Thus one might wonder whether $\Fprsim$ might
be a right Oka family. As it turns out, we are very fortunate and this is
in fact true in every ring!

\begin{proposition}
\label{Fprsim}
For any ring $R$, $\Fprsim(R)$ is an Oka family of right ideals.
\end{proposition}

\begin{proof}
We will denote $\F := \Fprsim(R)$. Because $I_R \sim R_R$ implies $I=R\in\Fpr$,
we see that $R\in\F$. Suppose that $I_R \subseteq R$ and $a\in R$
are such that $I+aR,\ a^{-1}I\in\F$. Set $C_1:=R/a^{-1}I$ and $C_2:=R/(I+aR)$,
so that we have an exact sequence
\[
0\rightarrow C_1\rightarrow R/I\rightarrow C_2\rightarrow 0.
\]
To prove that $I\in\F$, let $J_R\subseteq R$ be such that $R/J\cong R/I$. We
need to show that $J$ is principal. There is also an exact sequence
\[
0\rightarrow C_1\rightarrow R/J\rightarrow C_2\rightarrow 0.
\]
Thus there exists $x\in R$ with $C_1\cong(J+xR)/J$ and $C_2\cong R/(J+xR)$. But
then $R/(I+aR)=C_2\cong R/(J+xR)$ and $I+aR\in\F$ imply that $J+xR=cR$ for some
$c\in R$. Now
\[
\frac{R}{a^{-1}I} = C_1 \cong \frac{J+xR}{J} = \frac{cR}{J} \cong \frac{R}{c^{-1}J}
\]
and $a^{-1}I\in\F$, so we find that $c^{-1}J$ is principal. Then $J\subseteq J+xR=cR$
gives $J=c(c^{-1}J)$, proving that $J$ is principal.
\end{proof}

The following elementary observation will be useful in a number of places. It is simply
a convenient restatement of the fact that $\Fprsim$ is the largest set of principal
right ideals that is closed under similarity.

\begin{lemma}
\label{when similar is principal lemma}
Let $\setS$ be a set of right ideals of a ring $R$ that is closed under similarity.
If $\setS\subseteq\Fpr$, then $\setS\subseteq\Fprsim$ (and, of course, conversely).
\end{lemma}


We are finally ready to state and prove our noncommutative generalization of
the Kaplansky-Cohen Theorem~\ref{original Kap Cohen Theorem}.

\begin{theorem}[A noncommutative Kaplansky-Cohen Theorem]
\label{Kap}
For any ring $R$, let $\setS$ be a right noetherian point annihilator set that
is closed under similarity. The following are equivalent:
\begin{itemize}
\item[\normalfont (1)] $R$ is a principal right ideal ring;
\item[\normalfont (2)] Every right ideal in $\setS$ is principal;
\item[\normalfont (3)] $\Fprsim$ is a right noetherian point annihilator set.
\end{itemize}
In particular, $R$ is a principal right ideal ring iff every cocritical right
ideal of $R$ is principal.
\end{theorem}

\begin{proof}
The set of cocritical right ideals of $R$ is a noetherian point annihilator
set that is closed under similarity, so it suffices to prove the equivalence of~(1)--(3).
It is easy to see that~(1) is equivalent to the claim that all right ideals lie in
$\Fprsim$. 
Also, it follows from Lemma~\ref{when similar is principal lemma} that~(2)
holds precisely when $\setS\subseteq\Fprsim$.
The equivalence of (1)--(3) now follows from Corollary~\ref{f.g. "PIP supplement"}
and Proposition~\ref{Fprsim}.
\end{proof}

As with Cohen's Theorem, there exist previous noncommutative generalizations
of the Kaplansky-Cohen theorem in the literature. In~\S\ref{previous generalizations section}
we relate our theorem with these earlier results.

Comparing our two versions of the Kaplansky-Cohen Theorem, we see that
Corollary~\ref{weak Kap} follows from Theorem~\ref{Kap}, at least if we consider
condition~(3) in each equivalence. (Recall Remark~\ref{larger point annihilator sets},
and the fact that $\Fprsim\subseteq\Fpr$.)
However, this does not mean that Corollary~\ref{weak Kap} is obsolete. It is clear
that Theorem~\ref{Kap} is preferable to Corollary~\ref{weak Kap} if we have
enough knowledge about the point annihilator set $\setS$ but we do not know whether
the family $\Fpr$ is closed under similarity. On the other hand, if we are working in a
class of rings for which we know that $\Fpr$ is closed under similarity, then
Corollary~\ref{weak Kap} may be of more use. This proved to be the case in
Corollary~\ref{left perfect PRIR}, where we were able to reduce the point annihilator
set $\setS$ to a finite set.

Notice that our earlier examination of the Weyl algebra $A_1(k)$ in
Example~\ref{Fpr is not Oka} fits nicely with Theorem~\ref{Kap}, because the
nonprincipal right ideal discussed in that example was shown to be cocritical.

\separate

As a simple application of Theorem~\ref{Kap}, we can show that \emph{a domain
$R$ with right Krull dimension $\leq 1$ is a principal right ideal domain iff its maximal right
ideals are principal.} Indeed, by Proposition~\ref{cocritical ideals} the zero
ideal of $R$ is 1-cocritical as a right ideal (and it is, of course, principal).
Thus any nonzero cocritical right ideal of $R$ is 0-critical and therefore is
a maximal right ideal. The claim then follows from Theorem~\ref{Kap}. However,
we will prove a substantially more general version of this fact in
Proposition~\ref{when semiprime is PRIR}.

\section{Families closed under direct summands}
\label{essential right ideal section}

In this section we will develop further generalizations of Cohen's Theorem and
the Kaplansky-Cohen theorems by further reducing the set of right ideals in a ring
which we are required to ``test.'' In particular, where our previous theorems
stated that it was sufficient to check that every right ideal in some noetherian
point annihilator set $\setS$ is finitely generated (or principal), we will
further reduce the task to checking that every \emph{essential} right ideal in $\setS$
is finitely generated (or principal). We begin with a definition, temporarily
digressing to families of submodules of a given modules other than $R_R$.

\begin{definition}
Let $M_R$ be a module over a ring $R$. We will say that a family $\F$ of
submodules of $M$ is \emph{closed under direct summands} if for any $N\in\F$,
any direct summand of $N$ also lies in $\F$.
\end{definition}

Notice that a family $\F$ of submodules of $M$ that is closed under direct
summands necessarily has $0\in\F$ as long as $\F\neq\varnothing$. The
following result is the reason for our interest in families that are closed
under direct summands. It shows the link between such families and
the essential submodules of $M$. 

\begin{lemma}
\label{essentials suffice}
In a module $M_R$, let $\F$ be a family of submodules that is closed under
direct summands. Then all submodules of $M$ lie in $\F$ iff all essential
submodules of $M$ lie in $\F$.
In particular, if $\F$ is a family of right ideals in a ring $R$ that is
closed under direct summands, then all right ideals of $R$ lie in $\F$ iff
all essential right ideals of $R$ lie in $\F$.
\end{lemma}

\begin{proof}
(``If'' direction) Suppose that every essential submodule of $M$ lies in $\F$,
and let $L_R\subseteq M$. By Zorn's lemma there exists a submodule $N_R$
maximal with respect to $L\cap N=0$ (in the literature, such $N$ is referred
to as a \emph{complement} to $L$). We claim that $N\oplus L$ is an essential
submodule of $M$. Indeed, assume for contradiction that $0\neq K\subseteq M$
is a submodule such that $(L\oplus N)\cap K=0$. Then we have the direct sum
$L\oplus N\oplus K$ in $M$. It follows that $L\cap (N\oplus K)=0$, contradicting
the maximality of $N$. 

By assumption, $N\oplus L\subseteq_e M$ implies that $N\oplus L\in\F$. Then
because $\F$ is closed under direct summands, we conclude that $N\in\F$.
\end{proof}

With this result as our motivation, let us consider a few examples of families
of right ideals that are closed under direct summands.

\begin{example}
\label{summands example}
In any module $M_R$, the easiest nontrivial example of a family that is closed
under direct summands is the family $\F$ of all direct summands of $M$! The
application of Lemma~\ref{essentials suffice} in this case says that a module
$M$ is semisimple iff every essential submodule of $M$ is a direct summand.
However, it is easy to check that a direct summand of $M$ is essential in $M$
iff it is equal to $M$. So this says that a module is semisimple iff it has
no proper essential submodules. This is a known result; for instance,
see~\cite[Ex.~3.9]{ExercisesModules}.
\end{example}

\begin{example}
\label{number of generators example}
The family of finitely generated submodules of a module $M_R$ is certainly
closed under direct summands. It follows that a module $M$ is right noetherian
iff all of its essential submodules are finitely generated. Again, this fact
can be found, for instance, in~\cite[Ex.~6.11]{ExercisesModules}.

We can generalize the result above as follows. Let $\alpha$ be any cardinal
(finite or infinite), and let $\F$ be the family of all submodules of $M$
that have a generating set of size~$<\alpha$. Then $\F$ is again closed
under direct summands. So every submodule of $M$ is generated by~$<\alpha$
elements iff the essential submodules of $M$ are all generated by~$<\alpha$
elements.

Taking $M_R=R_R$ and $\alpha=2$, we see in particular that $\Fpr$ is closed under
direct summands,  and Lemma~\ref{essentials suffice} implies that $R$ is a PRIR
iff its essential right ideals are principal.
\end{example}

Here we end our digression into families of submodules of arbitrary modules
and focus our attention on families of right ideals in a ring $R$ that are
closed under direct summands. The next two examples are of a homological
nature.

\begin{example}
\label{essential Baer example}
For a module $M_R$, let $\F$ be the family of right ideals $I\subseteq R$
such that every module homomorphism $f \colon I \to M$ extends to a
homomorphism $R \to M$. This was shown to be a right Oka family
in~\cite[Prop.~5.16]{Reyes}. We claim that $\F$ is closed under direct
summands. For if $I \oplus J \in \F$ and $f \colon I \to M$ is any
homomorphism, then we may extend $f$ trivially to $I \oplus J \to M$.
This morphism in turn extends to $R \to M$ because $I \oplus J \in \F$.
Hence $I \in \F$.

By Baer's Criterion, every right ideal lies in $\F$ precisely when $M$
is injective. So applying Lemma~\ref{essentials suffice}, we find that $M$
is injective iff every essential right ideal of $R$ lies in $\F$. This
``essential version'' of Baer's Criterion has been noticed before; for instance,
see~\cite[Ex.~3.26]{ExercisesModules}.

More generally, for any module $M_R$ and integer $n\geq 0$, let $\F_M^n$
denote the family of right ideals $I\subseteq R$ such that $\Ext_R^{n+1}(R/I,M)=0$.
The family $\F$ above was shown to be equal to $\F_M^0$ in the proof
of~\cite[Prop.~5.16]{Reyes}. We claim that the families $\F^n$ are closed
under direct summands. The case $n=0$ is covered above, so suppose that
$n\geq 1$. Note that $\Ext_R^n(R,M)=\Ext_R^{n+1}(R,M)=0$ because $R_R$
is projective. So for any right ideal $K\subseteq R$, the long exact sequence
in $\Ext$ provides isomorphisms $\Ext_R^n(K,M)\cong\Ext_R^{n+1}(R/K,M)$.
Thus for any direct sum of right ideals $I\oplus J\subseteq R$, combining
this observation with a standard fact about $\Ext$ and direct sums gives
\begin{align*}
\Ext_R^{n+1}(R/(I\oplus J),M) &\cong \Ext_R^n(I\oplus J,M)\\
&\cong \Ext_R^n(I,M) \oplus \Ext_R^n(J,M)\\
&\cong \Ext_R^{n+1}(R/I,M) \oplus \Ext_R^{n+1}(R/J,M).
\end{align*}
This makes it clear that if $I\oplus J\in\F_M^n$, then $I\in\F_M^n$.

Extending Baer's Criterion, one can show that a module
$M_R$ has injective dimension~$\leq n$ iff $\Ext_R^{n+1}(R/I,M)=0$ for
all right ideals $I$ of $R$ (this is demonstrated in the proof
of~\cite[Thm.~8.16]{Rotman}).  If we apply Lemma~\ref{essentials suffice}
to the family $\F_M^n$, we see that for any module $M_R$ we have
$\id(M)\leq n$ iff $\Ext_R^{n+1}(R/I,M)=0$ for all essential right
ideals $I$ of $R$.\qed
\end{example}

\begin{example}
\label{pd example}
As an application of Example~\ref{essential Baer example} above, we produce
another example of a family that is closed under direct summands. Let
$\F^n$ be the family of all right ideals of $R$ such that $\pd(R/I)\leq n$.
Because $R/I$ has projective dimension~$\leq n$ iff $\Ext_R^{n+1}(R/I,M)=0$
for all modules $M$, we see that $\F^n$ is equal to the intersection
of all of the families $\F_M^n$ as $M$ ranges over all right
$R$-modules. Since all of these families are closed under direct
summands, $\F^n$ is also closed under summands. In this case we can
apply Lemma~\ref{essentials suffice} to say that a ring $R$ has
$\rgldim(R)\leq n$ iff $\pd(R/I)\leq n$ for all essential right ideals
$I_R\subseteq R$. Notice that when $n=0$,  $\F^0$ is the family of
right ideal direct summands mentioned in Example~\ref{summands example}.
\end{example}

Before continuing to the heart of this section, we require a small
observation as well as a new definition.

\begin{remark}
\label{remarks about essentials}
Notice that the set of essential right ideals is a divisible semifilter,
and is closed under similarity.
It is easy to see that the set is a semifilter. To see that it is
divisible, we will use the following fact about essential submodules:
for any homomorphism of modules $f:M_R\to N_R$ and any essential submodule
$N_0\subseteq N$, the preimage $f^{-1}(N_0)$ is an essential submodule of
$M$ (see~\cite[Ex.~3.7]{ExercisesModules} for a proof of this fact). 
Now given a right ideal $I\subseteq R$, $x^{-1}I$ is the preimage of the
right ideal $I$ under the homomorphism $R_R\to R_R$ given by left
multiplication by $x$. Thus if $I$ is an essential right ideal, so is
$x^{-1}I$.
Finally, to see that this set is closed under similarity, one only needs
to realize that $I_R\subseteq R$ is essential iff $R/I$ is a singular
module; see~\cite[Ex.~2(b)]{ExercisesModules}.
\end{remark}

\begin{definition}
\label{closure under summands}
Let $\F$ be a family of right ideals in a ring $R$. We define
\[
\widetilde{\F} := \{I_R\subseteq R : I\oplus J \in \F \text{ for some } J_R\subseteq R\}.
\]
This is the smallest family of right ideals containing $\F$ that is closed
under direct summands
\end{definition}

The next result, which is fundamental to this section, is a variation of
Theorem~\ref{"PIP supplement"} and Corollary~\ref{f.g. "PIP supplement"}.

\begin{theorem}
\label{essential "PIP supplement"}
Let $\F$ be an Oka family of right ideals in a ring $R$.
\begin{itemize}
\item[\normalfont (1)] Assume that every chain of right ideals in $\F'$ has an
upper bound in $\F'$, and let $\setS$ be a point annihilator set for the
class of modules $\{R/I:I\in\Max(\F')\}$. If every essential right ideal in
$\setS$ lies in $\F$, then all right ideals of $R$ lie in $\widetilde{\F}$.
\item[\normalfont (2)] Let $\setS$ be a noetherian point annihilator set for $R$,
and assume that $\F$ consists of finitely generated right ideals. If every
essential right ideal in $\setS$ lies in $\F$, then all right ideals of $R$
lie in $\widetilde{\F}$.
\end{itemize}
\end{theorem}

\begin{proof}
To prove (1), let $\setS$ and $\F$ satisfy the given hypotheses. Let
$\F_0$ denote the divisible semifilter of essential right ideals of
$R$. By assumption we have $\F_0\cap\setS\subseteq\F$, so it follows
from Theorem~\ref{"PIP supplement"} that $\F_0\subseteq\F$. Then all
essential right ideals of $R$ lie in $\widetilde{\F}\supseteq\F$, and
it follows from Lemma~\ref{essentials suffice} that all right ideals
lie in $\widetilde{\F}$.

Now (2) follows from (1) because the fact that $\F$ consists of finitely
generated right ideals implies both that every chain of right ideals
in $\F'$ has an upper bound in $\F'$ and that the class $\{R/I:I\in\Max(\F')\}$
consists of noetherian modules (as in the proof of Corollary~\ref{f.g. "PIP supplement"}).
\end{proof}

In particular, if the right Oka family $\F$ in the theorem above is in
fact closed under direct summands, then $\widetilde{\F}=\F$. Thus in this
case Theorem~\ref{essential "PIP supplement"} is a generalization of
Theorem~\ref{"PIP supplement"}. Our first application of this result will
be a strengthening of the noncommutative Cohen's Theorem~\ref{Cohen's Theorem}.

\begin{theorem}
\label{essential Cohen}
For a ring $R$, let $\setS$ be a right noetherian point annihilator set
(such as the set of cocritical right ideals). Then $R$ is right noetherian
iff every essential right ideal in $\setS$ is finitely generated.
\end{theorem}

\begin{proof}
(``If'' direction) This follows directly from Example~\ref{number of generators example}
and Theorem~\ref{essential "PIP supplement"}(2) by taking $\F=\widetilde{\F}$
to be the family of finitely generated right ideals of $R$.
\end{proof}

Our next application of Theorem~\ref{essential "PIP supplement"} will
strengthen our noncommutative version of the Kaplansky-Cohen Theorem~\ref{Kap}.
The careful statement of Theorem~\ref{essential "PIP supplement"} will pay off here.

\begin{theorem}
\label{essential Kap}
Let $R$ be a ring with noetherian point annihilator set $\setS$ that
is closed under similarity (such as the set of cocritical right ideals).
Then $R$ is a principal right ideal ring iff every essential right ideal
in $\setS$ is principal.
\end{theorem}

\begin{proof}
(``If'' direction) Suppose that every essential right ideal in $\setS$
is principal, and set $\F:=\Fprsim$. 
If $\setS_0\subseteq\setS$ is the set of essential right ideals
in $\setS$, then $\setS_0$ is closed under similarity because both $\setS$
and the set of essential right ideals are closed under similarity (recall
Remark~\ref{remarks about essentials}).
By hypothesis $\setS_0\subseteq\Fpr$, so Lemma~\ref{when similar is principal lemma}
gives $\setS_0\subseteq\Fprsim=:\F$. 
That is, every essential right ideal in $\setS$ lies in $\F$.
%
%
Now Theorem~\ref{essential "PIP supplement"}(2) implies that all right
ideals of $R$ lie in $\widetilde{\F}$.
But $\Fpr$ is closed under direct summands by Example~\ref{number of generators example},
so $\F\subseteq\Fpr$ implies that $\widetilde{\F}\subseteq\Fpr$.
Hence every right ideal of $R$ is principal.
%
%
\end{proof}

Our final applications of Theorem~\ref{essential "PIP supplement"} show how to
reduce the test sets for various homological properties in a right noetherian
ring.

\begin{theorem}
Let $R$ be a right noetherian ring, and let $\setS$ be a right (noetherian) point annihilator
set for $R$ (such as the set of cocritical right ideals).
\begin{itemize}
\item[\normalfont (1)] A module $M_R$ has injective dimension $\leq n$ iff $\Ext^{n+1}_R(R/P,M)=0$
for all essential right ideals $P\in\setS$.
\item[\normalfont (2)] Every finitely generated right $R$-module has finite projective
dimension iff, for every essential right ideal $P\in\setS$, one has $\pd(R/P)<\infty$.
\item[\normalfont (3)] $\rgldim(R)=\sup\{\pd(R/P):P\in\setS\textnormal{ is an essential right ideal}\}$.
\end{itemize}
\end{theorem}

\begin{proof}
For a module $M_R$ and a nonnegative integer $n$, let $\F^n_M$ and $\F^n$ be the
families introduced in Examples~\ref{essential Baer example} and~\ref{pd example},
where they were shown to be closed under direct summands. These familes were
shown to be right Oka families in~\cite[\S5.B]{Reyes}.
Defining $\F^{\infty}:=\bigcup_{n=1}^{\infty}\F^n$, it follows that $\F^{\infty}$ is
also a right Oka family that is closed under direct summands.

For part~(1), we note that a module $M_R$ has injective dimension $\leq n$ iff
$\F^n_M$ consists of all right ideals of $R$, which happens iff all essential right
ideals in $\setS$ lie in $\F^n_M$ according to Theorem~\ref{essential "PIP supplement"}(2).
Next we prove part~(2). Because every finitely generated right $R$-module is has a
finite filtration with cyclic filtration factors, and because the finiteness of projective
dimension is preserved by extensions, we see that every finitely generated right $R$-module
has finite projective dimension iff every cyclic right $R$-module does, iff $\F^{\infty}$
consists of all right ideals. By Theorem~\ref{essential "PIP supplement"}, this occurs iff
all essential right ideals in $\setS$ lie in $\F^{\infty}$.

Part~(3) similarly follows from Theorem~\ref{essential "PIP supplement"} applied to
the family $\F^n$, noting that $R$ has right global dimension $\leq n$ iff $\F^n$
consists of all right ideals.
\end{proof}

The above joins a whole host of results stating that certain homological properties
can be tested on special sets of ideals. We mention only a few relevant references here.
When $R$ is commutative, $\setS=\Spec(R)$, and $n=0$ in part~(1), the theorem above
recovers a result of J.\,A.~Beachy and W.\,D.~Weakley in~\cite{BeachyWeakley}.  
Part~(2) generalizes a result characterizing commutative regular rings, the ``globalizations''
of regular local rings (see~\cite[(5.94)]{Lectures}).
Many results along the lines of part~(3) are known.
For instance, a result of J.\,J.~Koker in~\cite[Lem.~2.1]{Koker} implies that if a ring $R$
has right Krull dimension, then its right global dimension is equal to the supremum of the
projective dimensions of the right modules $R/P$, where $P$ ranges over the cocritical right
ideals of $R$.
On the other hand, for a commutative noetherian ring $R$ the global dimension of $R$ is
equal to the supremum of $\pd(R/\m)$, where $\m$ ranges over the maximal ideals of $R$
(see~\cite[(5.92)]{Lectures}).
It has also been shown by K.\,R.~Goodearl~\cite[Thm.~16]{Goodearl} and
S.\,M.~Bhatwadekar~\cite[Prop.~1.1]{Bhatwadekar} 
that for a (left and right) noetherian ring $R$ whose global dimension is finite, the global
dimension of $R$ is the supremum of $\pd(R/\m)$ where $\m$ ranges over the maximal right
ideals of $R$. It is an open question whether the finiteness of the global dimension can be
dropped~\cite[Appendix]{GoodearlWarfield}.

\section{A noncommutative generalization of Kaplansky's Theorem}
\label{Kap's theorem section}

The goal of this section is to prove a noncommutative generalization of
Kaplansky's Theorem~\ref{original Kap's Theorem}.
Specifically, we shall show in Theorem~\ref{noetherian Kap} that
\emph{a noetherian ring whose maximal right ideals are all principal
is a principal right ideal ring}. To motivate our approach, we shall
recall a result~\cite[Theorem~C]{GoldiePIR} of A.\,W.~Goldie: a left
noetherian principal right ideal ring is a direct sum of a semiprime
ring and an artinian ring. 
Inspired by this fact, our proof of Theorem~\ref{noetherian Kap}
will proceed by taking noetherian ring whose maximal right ideals are
principal and decomposing it as a direct sum of a semiprime ring and
an artinian ring.
This should seem reasonable because we have already shown
in Corollary~\ref{left perfect PRIR} that, in order to test whether an
artinian ring is a PRIR, it suffices to test only its maximal right ideals.

With Goldie's result in mind, we begin this section by investigating
under what conditions one can check the PRIR condition on a semiprime
ring by testing only its maximal right ideals.
The first result applies to semiprime rings with small right Krull dimension.

\begin{proposition}
\label{when semiprime is PRIR}
Let $R$ be a semiprime ring with $\rKdim(R)\leq 1$. Then $R$ is a
principal right ideal ring iff its maximal right ideals are principal.
\end{proposition}

\begin{proof}
(``If'' direction) By Theorem~\ref{essential Kap}, it suffices to show that
the essential cocritical right ideals of $R$ are principal. Thus it is
enough to show that every essential cocritical right ideal of $R$ is
maximal. According to~\cite[6.3.10]{McConnellRobson} the fact that $R$
is semiprime with right Krull dimension means that, for every
$E_R\subseteq R$, $\Kdim(R/E)<\Kdim(R_R)=1$. 
So $\Kdim(R/E)\leq 0$, and if $E$ is also cocritical then it is
0-cocritical and thus is maximal. This completes the proof.
\end{proof}

In Example~\ref{local right noeth domain example} below we will show that
the hypothesis on the right Krull dimension cannot be relaxed.
Of course, it is not the case that \emph{every} semiprime PRIR has right
Krull dimension~$\leq 1$. In fact, in~\cite[Ex.~10.3]{GordonRobson} it is
shown (using a construction of A.\,V.~Jategaonkar from~\cite{Jategaonkar}) that there exist principal
right ideal domains whose right Krull dimension is equal to any prescribed
ordinal! So while Proposition~\ref{when semiprime is PRIR} gives a sufficient
condition for semiprime rings to be PRIRs, it is certainly not a necessary
condition. However, with some additional effort we will use this result to
formulate a precise characterization of semiprime left \emph{and} right
principal ideal rings in Corollary~\ref{semiprime PIR} below.

We will show in Proposition~\ref{semiprime must have Kdim 1}
below that if a semiprime ring with a certain finiteness condition on the
left has all maximal right ideals principal, it must have small right Krull
dimension.
We take this opportunity to recall that a multiplicatively
closed subset $S\subseteq R$ is \emph{saturated} if, for any $a,b\in R$,
$ab\in S$ implies $a,b\in S$. 

\begin{lemma}
\label{regular element has finite colength}
Let $R$ be a ring in which the multiplicative set of (resp.\ left) regular elements
is saturated and which satisfies the ascending chain condition on left ideals of
the form $Rs$ where $s\in R$ is a (resp.\ left) regular element.
Furthermore, suppose that every maximal right ideal of $R$ is principal.
If $b\in R$ is a (resp.\ left) regular element, then $R/bR$ has finite length.
\end{lemma}

\begin{proof}
This argument adapts some of the basic ideas of factorization in
noncommutative domains, as in Prop.~0.9.3 and Thm.~1.3.5 of~\cite{Cohn}.
However, we do not assume any of those results here.

If our fixed $b\in R$ is not right invertible, then $bR\neq R$. If $bR$
is not maximal, choose a maximal right ideal $a_1 R\subsetneq R$ such that
$bR\subsetneq a_1 R$. Then $b=a_1 b_1$ for some $b_1\in R$.
We claim that $Rb\subseteq Rb_{1}$ is strict. Indeed, assume for
contradiction that $Rb = Rb_1$. Then we may write $b_{1}=ub$ for some
$u\in R$. Thus $b=a_1 b_1=a_1 u b$, and because $b$ is (left) regular
we have $a_1 u=1$. This contradicts the fact that $a_1 R$ is maximal.
Hence $Rb\subsetneq Rb_1$.

Because the set of (left) regular elements is saturated, we may
now replace $b$ above by $b_1$ and proceed inductively to
write $b_{i-1}=a_i b_i$ (if $b_{i-1}R$ is not maximal) where
$b_i$ is (left) regular and $a_i R$ is a maximal right ideal. By
the ACC condition on $R$, the chain
\[
Rb \subsetneq Rb_1 \subsetneq Rb_2 \subsetneq \cdots
\]
cannot continue indefinitely. So the process must terminate,
say at $b_{n-1}=a_n b_n$. This means that $b_n R$ is
a maximal right ideal. Writing $a_{n+1}:=b_n$, we have a
factorization $b=a_1 \cdots a_{n+1}$ where the right ideals
$a_i R$ are maximal. Then in the filtration
\[
bR = (a_{1} \cdots a_{n+1})R \subseteq (a_{1}\cdots a_{n})R \subseteq
	\cdots \subseteq a_{1}R \subseteq R,
\]
each factor module $(a_1 \cdots a_{j-1})R/(a_{1}\cdots a_{j})R$
is a homomorphic image of the simple module $R/a_j R$ (via left
multiplication by $a_1 \cdots a_{j-1}$) and thus is simple. This
proves that $R/bR$ has finite length, as desired.
\end{proof}

In light of the hypotheses assumed above, the following definition
will be useful.

\begin{definition}
We will say that a ring $R$ satisfies \emph{left ACC-reg} if it
satisfies the ascending chain condition on left ideals of the form
$Rs$ where $s\in R$ is a regular element.
\end{definition}

\begin{proposition}
\label{semiprime must have Kdim 1}
Let $R$ be a semiprime ring with right Krull dimension that satisfies
left ACC-reg. If all of the maximal right ideals of $R$ are principal,
then $\rKdim(R)\leq 1$ and $R$ is a principal right ideal ring.
\end{proposition}

\begin{proof}
Because $R$ is semiprime and has right Krull dimension, it is right
Goldie (see~\cite[6.3.5]{McConnellRobson}). This has two important
consequences.
First, the set of regular elements of $R$ is saturated (because it is the
intersection of $R$ with the group of units in its semisimple right ring of
quotients).
Second, the essential right ideals of $R$ are precisely the right ideals
containing a regular element (see~\cite[(11.13)]{Lectures}).
Thus, for every $E_R\subseteq_e R$, $R/E$ has finite length by
Lemma~\ref{regular element has finite colength} and thus has Krull
dimension at most~0. Now~\cite[6.3.10]{McConnellRobson} provides
us with the following equation for $\rKdim(R)$ (which is valid because $R$
is semiprime with right Krull dimension):
\[
\rKdim(R) = \sup\{ \Kdim(R/E)+1 : E_R \subseteq_{e} R \} \leq 1. 
\]
Applying Proposition~\ref{when semiprime is PRIR}, we see that $R$ is a
principal right ideal ring.
\end{proof}

An immediate consequence is the aforementioned characterization of semiprime PIRs.

\begin{corollary}
\label{semiprime PIR}
Let $R$ be a semiprime ring.
\begin{itemize}
\item[\normalfont (1)] $R$ is a principal ideal ring iff its left and
right Krull dimensions are both at most~1 and the maximal left ideals
and maximal right ideals of $R$ are all principal.
\item[\normalfont (2)] Suppose that $R$ satisfies left ACC-reg. Then
$R$ is a principal right ideal ring iff $\rKdim(R)\leq 1$ and the maximal
right ideals of $R$ are principal.
\end{itemize}
\end{corollary}


It is possible to strengthen Proposition~\ref{semiprime must have Kdim 1}
to show that more general types of rings must have small right Krull
dimension.

\begin{corollary}
\label{rings which must have Kdim 1}
Let $R$ be a ring with right Krull dimension, and let $N$ be its
prime radical. Suppose that one of the following two conditions
holds:
\begin{itemize}
\item[\normalfont (A)] $R/N$ satisfies left ACC-reg;
\item[\normalfont (B)] $R/P$ satisfies left ACC-reg for every minimal
prime ideal $P\lhd R$.
\end{itemize}
If the maximal right ideals of $R$ are principal, then $\rKdim(R)\leq 1$.
In particular, a noetherian ring whose maximal right ideals are
principal has right Krull dimension at most~1.
\end{corollary}

\begin{proof}
According to~\cite[6.3.8]{McConnellRobson}, the ring $R$ with
right Krull dimension has finitely many minimal prime ideals
$P_1,\dots,P_n$ and
\[
\rKdim(R) = \rKdim(R/N) = \max\{ \rKdim(R/P_i) \}.
\]
Because every factor ring of $R$ again has principal maximal right
ideals, we may now apply Proposition~\ref{semiprime must have Kdim 1}.
\end{proof}

It is an open question whether the left and right Krull dimensions of a general
noetherian ring must be equal~\cite[Appendix]{GoodearlWarfield}.
However, another application of Proposition~\ref{semiprime must have Kdim 1} shows
that the Krull dimension of a noetherian PRIR must is symmetric. 

\begin{corollary}
A left noetherian principal right ideal ring $R$ has
\[
\lKdim(R) = \rKdim(R) \leq 1.
\]
\end{corollary}

\begin{proof}
As mentioned before, the Krull dimension of $R$ is not changed upon factoring out its
nilradical~\cite[6.3.8]{McConnellRobson}; thus we may assume that $R$ is semiprime.
In this case, a result of J.\,C.~Robson~\cite[Cor.~3.7]{RobsonPrincipal} states that
because $R$ is a noetherian PRIR, it must also be a PLIR.
According to Proposition~\ref{semiprime must have Kdim 1}, both $\lKdim(R)$ and
$\rKdim(R)$ are at most~1. Now $R$ has Krull dimension~0 on either side precisely when
$R$ is artinian on that side. But a noetherian ring is artinian on one side iff it is artinian on
the other side. (This follows, for instance, from the Hopkins-Levitzki Theorem~\cite[(4.15)]{FC}.)
Thus we see that the left and right Krull dimensions of $R$ must coincide, both equal to~0
when $R$ is artinian and both equal to~1 when $R$ is not artinian.
\end{proof}

The next preparatory result provides a method of testing whether
a module over a semilocal ring is zero. One may think of this as a variation
of Nakayama's Lemma (even though the latter is used in the proof below).

\begin{lemma}
\label{semilocal lemma}
Let $R$ be a semilocal ring, and let ${}_{R}B$ be a finitely
generated left module. If $B=\m B$ for all maximal right ideals
$\m$ of $R$, then $B=0$.
\end{lemma}

\begin{proof}
Let $R$ and ${}_{R}B$ be as above, and let $J=\rad(R)$. We claim
that $B/JB$ satisfies the same hypotheses over the semisimple ring
$R/J$. Indeed, the maximal right ideals of $R/J$ are the right
ideals of the form $\m/J$ for a maximal right ideal $\m$ of $R$.
For such $\m/J$ we have
\[
(\m / J) \cdot (B / JB) = \m B / JB = B / JB.
\]
Also, $B/JB$ is finitely generated over $R/J$. So $B/JB$ indeed satisfies
the same hypotheses over $R/J$.  If we knew the lemma to hold over all
semisimple rings, it would follow that $B/JB=0$. Nakayama's Lemma would
then imply that $B=0$.

So we may assume that $R$ is semisimple. Choose orthogonal
idempotents $e_1,\dots,e_n$ in $R$ whose sum is $1$ such that
$R_R=\bigoplus e_iR$ is a decomposition of $R$ into minimal right ideals.
Then for any $k$, $(1-e_k)R=\bigoplus_{i\neq k}e_iR$ is a maximal right
ideal of $R$. By hypothesis, we have $B=(1-e_k)RB=(1-e_k)B$. Because the
$e_i$ are orthogonal,
\[
(1-e_1) \cdots (1-e_k) = 1 - (e_1 + \cdots + e_k)
\]
In particular, $(1-e_1)\cdots(1-e_n)=1-(e_1+\cdots +e_n)=0$. It follows that 
\[
B = (1-e_1)B = (1-e_1)(1-e_2)B = \cdots = (1-e_1) \cdots (1-e_n)B = 0.\qedhere
\]
\end{proof}

Let us review some relevant results on noetherian rings. For an
ideal $I$ of a ring $R$, we let $\C(I)$ denote the set of elements
$c \in R$ such that $c+I$ is a regular element of $R/I$. A theorem
of J.\,C.~Robson~\cite{Robson} states that a noetherian ring $R$ with
prime radical $N$ is a direct sum of a semiprime ring and an artinian
ring iff, for every $c \in \C(N)$, $N=cN=Nc$. However, Robson commented
in~\cite[p.~346]{Robson} that if one only assumes that $N=cN$ for all
$c \in \C(N)$, one can still conclude that there exists an idempotent
$e\in R$ such that $eRe$ is semiprime, $(1-e)R(1-e)$ is artinian, and
$eR(1-e)=0$. This gives a useful ``triangular decomposition'' of such
a ring. In particular it can be used to derive the result of Goldie,
mentioned at the beginning of this section, that a left noetherian
principal right ideal ring is a direct sum of a semiprime ring and
an artinian ring. The first paragraph of our argument below borrows
from the proof of this last statement given in~\cite[Thm.~4]{Robson}.

\separate

With all of the above results and remarks in hand, we are finally ready to
prove our noncommutative generalization of Kaplansky's
Theorem~\ref{original Kap's Theorem}.

\begin{theorem}[A noncommutative Kaplansky's Theorem]
\label{noetherian Kap}
A noetherian ring is a principal right ideal ring iff its maximal
right ideals are principal.
\end{theorem}

\begin{proof}
(``If'' direction) Suppose $R$ is a noetherian ring whose maximal
right ideals are principal. Notice that every factor ring of $R$
satisfies the same hypotheses. Let $N\lhd R$ be the prime radical
of $R$. We claim that $N=cN$ for every $c\in\C(N)$. Let $x\mapsto\bar{x}$
denote the canonical map $R \to R/N=:\overline{R}$. By Proposition~\ref{semiprime must have Kdim 1}, 
$\rKdim(\overline{R})\leq 1$. For $c \in \C(N)$, the element
$\bar{c}\in\overline{R}$ is regular. So by~\cite[6.3.9]{McConnellRobson}
we must have $\Kdim(\overline{R}/\bar{c}\overline{R})<\Kdim(\overline{R})\leq 1$.
So the right $R$-module $R/(N+cR)\cong\overline{R}/\bar{c}\overline{R}$ has
Krull dimension at most~0 and thus has finite length. Hence $R/(N+cR)$
has a finite filtration with factors isomorphic to $R/\m_i$ for some
maximal right ideals $\m_1,\dots,\m_p$ of $R$. The set of maximal right ideals
of $R$ is certainly closed under similarity (it is the set of right ideals whose factor
module is simple), so by Lemma~\ref{when similar is principal lemma} all maximal
right ideals lie in the right Oka family $\Fprsim$. It follows from
Corollary~\ref{repeated extensions} that we have $N+cR\in\Fprsim$. Choose
$d \in R$ such that $N+cR=dR$.  Now in $\overline{R}$,
$\bar{c}\overline{R} = \bar{d}\overline{R}$ means that $\bar{c}=\bar{d}\bar{r}$
for some $r\in R$. Because the set of regular elements in the semiprime noetherian
ring $\overline{R}$ is saturated, the fact that $c\in\C(N)$ implies that $d\in\C(N)$.
Now $N \subseteq dR$ implies that $N=d(d^{-1}N)$, and $d\in\C(N)$
gives $d^{-1}N=N$. Thus $N=d(d^{-1}N)=dN=(cR+N)N=cN+N^2$, and we
conclude from Nakayama's Lemma~\cite[(4.22)]{FC} (or by induction and the
fact that $N$ is nilpotent) that $N=cN$.

Now according to Robson's decomposition result~\cite[p.~346]{Robson}
the ring $R$ is (up to isomorphism) of the form
\[
R =
\begin{pmatrix}
A & B\\
0 & S
\end{pmatrix},
\]
where $A$ is an artinian ring, $S$ is a semiprime ring, and
${}_{A}B_S$ is a (left and right noetherian) bimodule. Given any
maximal right ideal $\m$ of $A$, we will show that $B=\m B$. The
following is a maximal right ideal of $R$, and is therefore principal:
\[
\begin{pmatrix}
\m & B\\
0 & S
\end{pmatrix} =
\begin{pmatrix}
x & y\\
0 & z
\end{pmatrix} \cdot R
\]
for some $x\in\m$, $y\in B$, and $z\in S$. It is easy to see
that $zS=S$. Because $S$ is noetherian, $z$ must be a unit.
Now for any $\beta\in B$, there exists $\left(
\begin{smallmatrix}
a & b\\
0 & c
\end{smallmatrix}
\right)\in R$ such that 
\[
\begin{pmatrix}
x & y\\
0 & z
\end{pmatrix}
\begin{pmatrix}
a & b\\
0 & c
\end{pmatrix} =
\begin{pmatrix}
0 & \beta\\
0 & 0
\end{pmatrix} \in 
\begin{pmatrix}
\m & B\\
0 & S
\end{pmatrix}.
\]
Since $zc=0$ and $z$ is a unit, we must have $c=0$. Thus
$\beta=xb\in\m B$. Because $\beta\in B$ was arbitrary, this
proves that $B=\m B$. Since this holds for every maximal right
ideal $\m$ of $A$, we conclude from Lemma~\ref{semilocal lemma}
that $B=0$.

Hence $R=A\oplus S$ where $A$ is an artinian ring and $S$ is
a semiprime ring. The maximal right ideals of both $S$ and $A$
must also be principal. The artinian ring $A$ is a PRIR according
to Corollary~\ref{left perfect PRIR}, and it follows from
Proposition~\ref{semiprime must have Kdim 1} that the semiprime
ring $S$ is a PRIR.
It follows that $R=A\oplus S$ is a PRIR.
\end{proof}

It is interesting to notice that, in the commutative setting, Kaplansky's
Theorem~\ref{original Kap's Theorem} is ``stronger'' than the Kaplansky-Cohen
Theorem~\ref{original Kap Cohen Theorem}, in the sense that Kaplansky originally
derived Theorem~\ref{original Kap Cohen Theorem} as a consequence of
Theorem~\ref{original Kap's Theorem}.
This is opposite from our present situation, where the noncommutative version
of the Kaplansky Theorem~\ref{noetherian Kap} in fact follows from (the ``essential
version'' of) the noncommutative Kaplansky-Cohen Theorem~\ref{essential Kap} (through
a series of other intermediate results).

The following example shows that Kaplansky's Theorem does
not generalize if we remove the left noetherian hypothesis.

\begin{example}
\label{local right noeth example}
\emph{A local right noetherian ring $R$ with right Krull dimension 1 whose unique
maximal right ideal is principal, but which is not a principal right ideal ring.} 
This construction is based on an exercise given in~\cite[Ex.~19.12]{ExercisesClassical}.
Let $k$ be a field such that there exists a field isomorphism $\theta\colon k(x)\to k$
(which certainly does \emph{not} fix $k$), such as $k=\mathbb{Q}(x_1, x_2,\dots)$. 
Consider the discrete valuation ring $A=k[x]_{(x)}$.
Given a finitely generated module $M_A$, we define a ring
$R:=A\oplus M$ with multiplication given by
\[
(a,m) \cdot (a',m') := (aa',m'\theta(a) + ma').
\]
Let $\m=xA \oplus M$ and $N=0\oplus M$, both of
which are ideals of $R$. Notice that $N^2=0$ while $\bar{R}:=R/N\cong A$
is a domain. This means that $N$ is the prime radical of $R$. Thus
$N$ is contained in the Jacobson radical $\rad(R)$. Because
$R/\rad(R)\cong\bar{R}/\rad(\bar{R})$ is a field, the ring $R$ is
local with Jacobson radical equal to $\m$.
Using the fact that $\theta(x)\in k$ is a unit in $A$, it is easy to
conclude that $\m=(x,0)\cdot R$ is a principal right ideal.

Next we show that $R$ is right noetherian. Because the ring
$R/N\cong A$ is noetherian, it is noetherian as a right $R$-module.
Also, because $N^2=0$, the right $R$-action on $N_R=(0 \oplus M)_R$
factors through $R/N\cong A$. Because $A$ is noetherian and $M_A$
is finitely generated, this means that $N_R$ is noetherian. So
$R_R$ is an extension of the noetherian right modules $R/N$ and
$N$, proving that $R$ is right noetherian. Because the prime
radical of $R$ is $N$, $\rKdim(R)=\rKdim(R/N)=\Kdim(A)=1$ (see~\cite[6.3.8]{McConnellRobson}).
Finally, because the $R$-action on $N_R=(0 \oplus M)_R$ factors
through $R/N \cong A$, if $M_A$ is any noncyclic $A$-module
then $N$ is not principal as a right ideal in $R$. In fact,
because the minimal number of generators of $\mu(N_R)$ is equal
to $\mu(M_A)<\infty$, this number can be made as large as one
desires.
\end{example}

Notice that the example above is not semiprime, in accordance with
Proposition~\ref{when semiprime is PRIR}. With some extra work, we can
produce a similar example $R$ that is a domain.
By Proposition~\ref{when semiprime is PRIR} again, we expect such $R$ to
have right Krull dimension~$>1$.
(We thank G.\,M.~Bergman for helping to correct an earlier, incorrect version
of this example.)

\begin{example}
\label{local right noeth domain example}
\emph{A local right noetherian domain $R$ with right Krull dimension~2 whose
unique maximal right ideal is principal, but which is not a principal right
ideal ring.}
Let $k$, $\theta\colon k(x)\overset{\sim}{\to}k$, and $A=k[x]_{(x)}$ be as in
Example~\ref{local right noeth example}.
Let $B=A[[y; \theta]]\supseteq A$, the ring of skew power series over $A$
subject to the relation $ay=y\theta(a)$. Consider the ideal $I=y^2 B$, and
define the subring $R:=A\oplus I\subseteq B$.
(Notice that $R$ is the subring of $B$ consisting of power series in which
$y$ does not appear with exponent~1. We can suggestively write
$R=A[[y^2, y^3; \theta]]$, with the understanding that the equation
$ay=y\theta(a)$ only has meaning via its consequences $ay^n=y\theta^n(a)$
for $n\geq 2$.)
Being a subring of the domain $B$, $R$ itself is a domain.

We claim that $I\subseteq\rad(R)$. It suffices to show that
$1+I \subseteq U(R)$ (see~\cite[(4.5)]{FC}). Let $i\in I$; then $1+i$ is a unit
of $B$ because $I\subseteq yB=\rad(B)$. For $i':=-(1+i)^{-1}i=-i(1+i)^{-1}\in I$
(note: $(1+i)^{-1}$ commutes with $i$ because $1+i$ does), we have
\[
(1+i) \cdot (1+i') = 1+ i + (1+i)i' = 1,
\]
and similarly $(1+i')(1+i)=1$. So $1+i\in U(R)$ as desired. One can now proceed as
in Example~\ref{local right noeth example} to show that $R$ is a local ring whose
unique maximal right ideal $\m:=xA\oplus I=xR$ is principal.

It is easy to see that $R$ is a free right module over the subring
$A[[y^2; \theta]]\cong A[[t; \theta]]=:S$ with basis $\{1,y^3\}$.
Because $S$ is right noetherian (in fact, a principal right ideal domain, according
to~\cite{Jategaonkar1}), $R$ is also right noetherian.
We claim that $\rKdim(S)=2$. First we show that for every
$f\in S\setminus\{0\}$, $\Kdim(S/fS)\leq 1$.
Indeed, we can write $f=t^m x^n u$ for some unit $u\in S$. It follows from
the filtration
\[
S \supseteq tS \supseteq t^2 S \supseteq \cdots \supseteq t^m S
\supseteq t^m xS \supseteq t^m x^2 S \supseteq \cdots \supseteq t^m x^n S = fS
\]
that $S/fS$ has a filtration whose factors are isomorhpic to either $S/tS\cong A$
or $S/xS\cong A/xA$.
These filtration factors have submodule lattices isomorphic to that of $A_A$
or $(A/xA)_A$, and thus respectively have Krull dimension~1 or~0.
Hence $\Kdim(S/fS)\leq 1$ as claimed.
Because $S$ has right Krull dimension and is a domain, we see from
Proposition~\ref{cocritical ideals} that $S_S$ is a critical module. We conclude
that $\rKdim(S)=2$.
Thus $\Kdim(R_S) = \Kdim(S_S^2) = \Kdim(S_S) = 2$ (the second equality follows
from the exact sequence $0\to S\to S^2\to S\to 0$), which implies that
$\Kdim(R_R)\leq\Kdim(R_S)=2$. On the other hand, the descending chain
$I\supseteq I^2 \supseteq I^3 \supseteq \cdots$ of right ideals in $R$ has
filtration factors $I^m/I^{m+1} = y^{2m}B/y^{2m+2}B\cong A\oplus A$.
These have Krull dimension~1, so we find $\Kdim(R_R)>1$ and thus $\rKdim(R)=2$.
%

Finally, we show that $I$ is not a principal right ideal of $R$. It suffices to show
that $I/I\m$ is not a cyclic right module over $R/\m\cong k$. 
Notice that $I\m= I(Ax+I)=Ix+I^2$.
Now $By\subseteq yB$ implies that $I^2 = (y^2 B)^2 = y^4 B$.
Also, $Ix = y^2 xA \oplus y^3 xA \oplus y^4 xA \cdots$. Thus
$I\m = Ix+I^2 = y^2 xA \oplus y^3 xA \oplus y^4 B$.
It follows that
\[
\frac{I}{I\m} \cong \frac{y^2 A[[y; \theta]]}{y^2 xA \oplus y^3 xA \oplus y^4 A[[y; \theta]]}
\cong y^2 k \oplus y^3 k
\]
is not a cyclic $k$-vector space, as desired.\qed
\end{example}

We conclude this section with some questions that arise in light of the results above.
Examples~\ref{local right noeth example} and~\ref{local right noeth domain example}
show that the left noetherian hypothesis in Theorem~\ref{noetherian Kap} cannot simply
be dropped.
While it seems somehow unnatural to try to omit the right noetherian hypothesis, we
have not found an example showing this to be impossible. Thus we ask the following.

\begin{question}
Does there exist a left (but not right) noetherian ring $R$ whose
maximal right ideals are all principal, but which is not a principal
right ideal ring? What if we assume, in addition, that $R$ has right
Krull dimension?
\end{question}

While reading an earlier draft of this work, G.\,M.~Bergman kindly pointed out
to us that no such example exists if we assume further that $R$ is a \emph{domain}.
We were able to generalize this to include semiprime right Goldie rings as follows.

\begin{proposition}
Let $R$ be a semiprime left noetherian ring in which every essential right
ideal contains a regular element (the latter hypothesis is satisfied if $R$
is a domain or if $R$ is right Goldie---in particular, if $R$ has right Krull
dimension). If every maximal right ideal of $R$ is principal, then $R$
is a principal right ideal ring.
\end{proposition}

\begin{proof}
By Example~\ref{number of generators example}, it is enough to show
that every essential right ideal of $R$ is principal.
To this end, fix $E_R\subseteq_e R$.
Because $R$ has a semisimple left ring of quotients, the multiplicative set
of regular elements of $R$ is saturated.
Thus the hypotheses of Lemma~\ref{regular element has finite colength} are
satsfied.
Since $E$ contains a regular element, that lemma implies that $R/E$ has finite
length.
So $R/E$ has a finite filtration whose factors are isomorphic to $R/\m_i$
for some maximal right ideals $\m_1, \dots, \m_n$ of $R$.
Since the set of maximal right ideals is closed under similarity,
Lemma~\ref{when similar is principal lemma} implies that all maximal right
ideals of $R$ lie in $\Fprsim$. Now Corollary~\ref{repeated extensions}
implies that $E\in\Fprsim\subseteq\Fpr$, so that $E$ is principal.
\end{proof}

We also ask to what extent the PRIR condition can be tested up to similarity.

\begin{question}
Suppose that $R$ is a noetherian ring each of whose maximal right ideals is
similar to a principal right ideal. Is $R$ a principal right ideal ring? If not,
then is every right ideal of $R$ similar to a principal right ideal?
\end{question}

It would be interesting to test the status of the first Weyl algebra $R:=A_1(k)$
with respect to this question. Is every maximal right ideal of $R$ similar to
a principal right ideal? Does $R$ have any right ideals that are not similar
to principal right ideals?
More generally, we wonder whether there exists \emph{any} ring that is not
a PRIR, but in which every right ideal is similar to a principal right ideal.

\section{Previous generalizations of the Cohen and Kaplansky theorems}
\label{previous generalizations section}

In this final section we will discuss how Theorem~\ref{Cohen's Theorem} and
Theorem~\ref{Kap} relate to earlier noncommutative generalizations of the
Cohen and Kaplansky-Cohen theorems in the literature. 
(We are not aware of any previous generalizations of Kaplansky's Theorem~\ref{original Kap's Theorem}.)
In~\cite{Koh}, K.~Koh generalized both of these theorems. He defined a right
ideal $I_R\subsetneq R$ to be a ``prime right ideal'' if, for any right ideals
$A,B\subseteq R$ such that $AI\subseteq I$, $AB\subseteq I$ implies that
$A\subseteq I$ or $B\subseteq I$.
Notice that this is equivalent to the condition that for $a,b\in R$,
$aRb\subseteq I$ with $aRI\subseteq I$ imply that either $a\in I$ or $b\in I$.
We will refer to such a right ideal as a \emph{Koh-prime right
ideal}. Koh showed that a ring $R$ is right noetherian (resp.\ a PRIR)
iff all of its Koh-prime right ideals are finitely generated (resp.\ principal).
Independently, in~\cite{Chandran1} (also appeared~\cite{Chandran2}) V.\,R.~Chandran
also gave generalizations of the Cohen and Kaplansky theorems, showing that a
right duo ring is right noetherian (resp.\ a PRIR) iff all prime ideals of $R$
are finitely generated (whether this is f.g.\ as an ideal or f.g.\ as a right
ideal is irrelevant, since $R$ is right duo). But Koh's result implies Chandran's
result, since a two-sided ideal is Koh-prime as a right ideal iff it is a prime
ideal in the usual sense.

Notice that our completely prime right ideals are necessarily Koh-prime right
ideals. For suppose that $P_R\subseteq R$ is completely prime and that $A,B\subseteq R$
are such that $AP\subseteq P$ and $AB\subseteq P$. If $A\nsubseteq P$, then there
exists $a\in A\setminus P$. Now $aP\subseteq P$, and for any $b\in B$ we have $ab\in P$.
It follows that $b\in P$ because $P$ is completely prime. So $B\subseteq P$, proving
that $P$ is Koh-prime. \emph{It follows that Theorem~\ref{Cohen's Theorem} and
Theorem~\ref{Kap}, with the set $\setS$ taken to be the set of completely prime
right ideals, imply Koh's theorems, which in turn imply Chandran's theorems.}

\separate

On the other hand, G.\,O.~Michler offered another noncommutative generalization
of Cohen's Theorem in~\cite{Michler}. He defined a right ideal $I\subsetneq R$ to
be ``prime'' if $aRb\subseteq I$ implies that either $a\in I$ or $b\in I$. This
is equivalent to saying that, for right ideals $A,B\subseteq R$, $AB\subseteq I$
implies that one of $A$ or $B$ lies in $I$. We will refer to such right ideals
as \emph{Michler-prime right ideals}. Michler proved in~\cite{Michler} that a ring
is right noetherian iff its Michler-prime right ideals are all finitely generated. 
Notice immediately that the Michler-prime right ideals of a given ring form a subset
of the set of all Koh-prime right ideals of that ring; thus Michler's version of Cohen's
Theorem generalizes Koh's version.

If we were to try to recover Michler's theorem directly from Theorem~\ref{Cohen's Theorem},
we would need to check that the Michler-prime right ideals form a noetherian
point annihilator set over an arbitrary ring $R$. In order to settle whether or
not this is true, we offer an alternate description of the Michler-prime right
ideals below. Recall that a module $M_R\neq 0$ is said to be a \emph{prime module}
if, for every nonzero submodule $N\subseteq M$, $\ann(N)=\ann(M)$. One can show
that the annihilator of a prime module is a prime ideal (for example, as in~\cite[(3.54)]{Lectures}).

\begin{proposition}
\label{Michler-prime characterization}
A right ideal $P\subsetneq R$ is Michler-prime iff $R/P$ is a prime module.
\end{proposition}

\begin{proof}
First suppose that $P$ is Michler-prime. To see that $R/P$ is a prime
module, consider a nonzero submodule $A/P\subseteq R/P$ (so that the right
ideal $A$ properly contains $P$). Denote $B:=\ann(A/P)\lhd R$. Then
$(A/P)\cdot B=0$ implies that $AB\subseteq P$. Because $P$ is Michler-prime,
this means that $B\subseteq P$, so that $(R/P)\cdot B=(P+B)/P=0$. So
$B=\ann(R/P)$, proving that the module $R/P$ is prime.

Conversely, suppose that $R/P$ is a prime module. Let $a,\ b\in R$
be such that $aRb\subseteq P$ and $a\notin P$. It follows that $b$
annihilates $(P+aR)/P\neq 0$, so that $b\in \ann((P+aR)/P)=\ann(R/P)$.
In particular, $(R/P)\cdot b=0$ implies that $b\in P$. This proves
that $P$ is Michler-prime.
\end{proof}

\begin{corollary}
\label{Michler-prime NPA}
For a ring $R$, the set $\setS$ of Michler-prime right ideals is a noetherian
point annihilator set iff every nonzero noetherian right $R$-module has a 
prime submodule. This is satisfied, in particular, if $R$ has the ACC on ideals.
\end{corollary}

\begin{proof}
The ``only if'' direction is clear from Proposition~\ref{Michler-prime characterization}.
For the ``if'' direction, let $M_R$ be any module with a prime submodule
$N$. Notice that a nonzero submodule of a prime module is prime. Thus for any
nonzero element $m\in N$, $R/\ann(m)\cong mR\subseteq N$ is a prime
module. By Proposition~\ref{Michler-prime characterization}, $\ann(m)$ is a
Michler-prime right ideal. So if every nonzero noetherian module has a prime
submodule, the set $\setS$ is a noetherian point annihilator set.

If $R$ satisfies ACC on ideals, then \emph{every} nonzero
right $R$-module has a prime submodule---see~\cite[(3.58)]{Lectures}.
So in this case $\setS$ is a point annihilator set, hence a noetherian point
annihilator set.
\end{proof}

We conclude from Corollary~\ref{Michler-prime NPA} and
Theorem~\ref{Cohen's Theorem} that \emph{a ring is right noetherian iff
it satisfies ACC on ideals and all of its Michler-prime right ideals are
finitely generated}. This is actually a slight generalization
of~\cite[Lem.~3]{Michler}, which Michler used as a ``stepping stone'' to
prove his main result.

Nevertheless, there do exist nonzero noetherian modules over some (large) rings
which do not have any prime submodules. Thus Michler's primes do not form a
noetherian point annihilator set in every ring. We include an example below.

\begin{example}
Let $k$ be a division ring, and let $R$ be the ring of
$\mathbb{N}\times\mathbb{N}$ row-finite upper triangular matrices over
$k$. Let $M_R=\bigoplus_{\mathbb{N}}k$, viewed as row vectors over $k$,
with the obvious right $R$-action. Let $M_i$ denote the submodule of $M$
consisting of row vectors whose first $i$ entries are zero. Then one can
show that
\[
M = M_0 \supsetneq M_1 \supsetneq M_2 \supsetneq \cdots
\]
are the only nonzero submodules of $M$. This visibly shows that $M$ is
noetherian. (Indeed, one can say more: every submodule of $M$ is actually
principal, generated by one of the ``standard basis vectors.'' We omit the
details because we will not use this fact.) However, one can see that
$\ann(M_i)$ is equal to the set of all matrices in $R$ whose first $i$ rows
are arbitrary and whose other rows are zero. So the fact that
\[
\ann(M_0) \subsetneq \ann(M_1) \subsetneq \ann(M_2) \subsetneq \cdots
\]
makes it clear that $M$ has no prime submodules.

Incidentally, $M_R$ is also an example of a cyclic 1-critical module
that is not a prime module. Thus, choosing a right ideal $I_R\subseteq R$ such
that $R/I\cong M$ (such as the right ideal of matrices in $R$ whose first row is
zero), we see that $I$ is cocritical but not Michler-prime. 
\end{example}

In spite of this complication, \emph{it is in fact possible to derive Michler's Theorem
from Theorem~\ref{Cohen's Theorem}}. The key observation that makes this possible
is a lemma~\cite[Lem.~2]{Smith1}  due to P.\,F.~Smith. This result states that if
every ideal of a ring $R$ contains a finite product of prime ideals each containing
that ideal, and if $R$ satisfies the ACC on prime ideals, then every nonzero right
$R$-module has a prime submodule.

\begin{theorem}[Michler]
\label{Michler's Theorem}
A ring $R$ is right noetherian iff all of the Michler-prime right ideals of $R$
are finitely generated.
\end{theorem}

\begin{proof}
(``If'' direction.) Suppose that the Michler-prime right ideals of $R$ are all
finitely generated. Every prime (two-sided) ideal of $R$ is Michler-prime, and
thus is finitely generated as a right ideal. By~\cite[Lemmas~4 \&~5]{Michler}
the following two conditions hold:
\begin{enumerate}
\item Every ideal $I\lhd R$ contains a product of finitely many prime ideals of
$R$, where each of these ideals contains $I$;
\item $R$ satisfies the ascending chain condition on prime ideals.
\end{enumerate}
It follows from~\cite[Lem.~2]{Smith1} that every nonzero right $R$-module
has a prime submodule. So by Corollary~\ref{Michler-prime NPA}, the set
of Michler-prime right ideals is a noetherian point annihilator set for $R$.
Now it follows from Theorem~\ref{Cohen's Theorem} that $R$ is a right noetherian
ring.
\end{proof}

In addition, our methods allow us to produce a generalization of the Kaplansky-Cohen
Theorem that is in the spirit of Michler's Theorem! Note that this was not proved
in~\cite{Michler}, and in fact seems to be a new result.

\begin{theorem}
\label{Michler Kap}
A ring $R$ is a principal right ideal ring iff all of the Michler-prime
right ideals of $R$ are principal.
\end{theorem}

\begin{proof}
(``If'' direction.) Suppose that all of the Michler-prime right ideals of $R$
are principal. As in the proof of Theorem~\ref{Michler's Theorem} above, it
follows that the set $\setS$ of Michler-prime right ideals of $R$ is a
noetherian point annihilator set for $R$. This set is closed under similarity
thanks to Proposition~\ref{Michler-prime characterization}, so Theorem~\ref{Kap}
implies that $R$ is a principal right ideal ring.
\end{proof}

For a given ring $R$, the effectiveness of Michler's Theorem versus Theorem~\ref{Cohen's Theorem}
with $\setS$ taken to be the set of completely prime right ideals of $R$
depends on the scarcity or abundance of right ideals in $R$ from the ``test
set'' in either theorem. For example, over a simple ring $R$, every nonzero
right $R$-module is certainly prime. So every proper right ideal of $R$ will
be Michler-prime. (In fact, Koh~\cite[Thm.~4.2]{Koh1} has shown even more:
a ring $R$ is simple iff all of its proper right ideals are Michler-prime.)
Thus for a simple ring $R$, Michler's theorem provides no advantage,
as we would still need to test \emph{every} right ideal to see whether $R$ is
right noetherian. On the other hand, all right ideals of a ring $R$ are completely
prime only if $R$ is a division ring (see~\cite[Prop.~2.11]{Reyes}). So outside of this
trivial class of rings, we are guaranteed that Theorem~\ref{Cohen's Theorem} with
$\setS=\{\text{completely prime right ideals}\}$ reduces the set of right ideals which
we need to test in order to determine whether a ring is right noetherian. We can expect
Theorem~\ref{Cohen's Theorem} to be increasingly effective when we take $\setS$ to
be either of the two smaller test sets in~\eqref{inclusion of NPA sets}.

\separate

There is another variant of Cohen's Theorem for right fully bounded rings. (Recall that
$R$ is right fully bounded if, for every prime ideal $P\lhd R$, every essential right ideal
of $R/P$ contains a nonzero ideal of $R/P$). This result says that \emph{a
right fully bounded ring is right noetherian iff all of its prime ideals are finitely
generated as right ideals}. A statement of this theorem is given in~\cite[p.~95]{Krause},
and it is attributed to G.\,O.~Michler and L.\,W.~Small independently. P.\,F.~Smith
provided a proof using homological methods in~\cite[Cor.~5]{Smith2} and an elementary
proof in~\cite[Thm.~1]{Smith3}.
(On a related note,~\cite{Smith3} also features a version of Cohen's Theorem for
modules over commutative rings.)
If one is satisfied to deal with the subclass of PI rings, then this result can be
proved via the approach of the present paper.

\begin{theorem}[Michler-Small]
A PI ring $R$ is right noetherian iff all of its prime ideals are finitely generated
as right ideals.
\end{theorem}

\begin{proof}
(``If'' direction.) Let $R$ be a PI ring in which every prime ideal is finitely
generated as a right ideal.
By~Theorem~\ref{Cohen's Theorem}, to prove that $R$ is right noetherian it is enough
to show that every nonzero noetherian right $R$-module $M_R$ has a nonzero cyclic
finitely presented submodule. (This basically produces a noetherian point annihilator
set of right ideals that are finitely generated.)

As in the proof of Theorem~\ref{Michler's Theorem}, every nonzero noetherian
right $R$-module has a prime submodule, so it suffices to assume that $M$ is prime.
In this case, $P:=\Ann(M)$ is prime.
By a result of Amitsur and Small~\cite[Prop.~3]{AmitsurSmall}, because the prime
PI ring $R/P$ has a faithful noetherian module $M$, it is a noetherian ring.
A result of Cauchon on right fully bounded noetherian rings (see, for
instance,~\cite[Thm.~9.10]{GoodearlWarfield}) now implies that there is an embedding
$R/P \hookrightarrow M^n$ for some integer $n\geq 1$, and we will identify $R/P$
with its image as a submodule of $M$.
Let $\pi \colon M^n \to M$ be the projection of $M^n$ onto one of its
components such that $R/P \nsubseteq \ker\pi$.
The module $(R/P)_R$ is finitely presented (because $P_R$ is finitely
generated) and noetherian.
Thus $\pi(R/P)\subseteq M$ is a factor of the finitely presented module $R/P$
by a finitely generated submodule and must istelf be finitely presented
(for instance, see~\cite[(4.26)(b)]{Lectures}).
Hence $M$ has a nonzero cyclic finitely presented submodule as desired.
\end{proof}

It is clear that the same proof would recover the Michler-Small Theorem for
right fully bounded rings (not only PI rings) if the following question has an
affirmative answer.

\begin{question}
\label{Michler-Small question}
Let $R$ is a prime right (fully) bounded ring with a faithful (prime) noetherian
module $M_R$. Is $R$ right noetherian, or equivalently, is there an embedding
$R_R \hookrightarrow M^n$ for some integer $n \geq 1$?
\end{question}

(The equivalence of the two questions follows from Cauchon's result, mentioned in
the proof above, that a prime right fully bounded noetherian ring with a faithful
module $M_R$ embeds into a finite direct sum of copies of $M$.)

\separate

In a more recent paper~\cite{Zabavskii}, B.\,V.~Zabavs'ki\u{\i} also studied noncommutative
versions of the Cohen and Kaplansky-Cohen theorems.  Theorem~1 of that paper states
that, for a right chain ring $R$ (i.e., a ring whose right ideals are totally ordered
under inclusion), if every Michler-prime right ideal is principal,
then $R$ is a principal right ideal ring. This is clearly generalized by Theorem~\ref{Michler Kap} above.
There is second version of the Kaplansky-Cohen Theorem in~\cite[Thm.~2]{Zabavskii}
using a test set that is equal to the set of Koh-prime right ideals. Thus this
theorem is equivalent to Koh's theorem.
A noncommutative Cohen's Theorem is proved in~\cite[Thm.~5]{Zabavskii} using a test
set that contains the Michler-prime right ideals as  a subset; thus this result is
subsumed by Michler's Theorem.
Finally, there are also some results in~\cite{Zabavskii} investigating when every
two-sided ideal of a ring is either finitely generated or principal when considered as
a right ideal.


\section*{Acknowledgments}
I wish to thank T.\,Y.~Lam for his guidance throughout the time that I worked on the
topic at hand. He provided comments on a number of drafts of this work and helped to fix
some of the terminology introduced here.
I am grateful to George Bergman who read two drafts of this work and made many
useful comments. In particular, he provided comments that helped to clarify the content
of~\S\ref{point annihilator theorem section}, as well as an effective suggestion to help
repair an error in Example~\ref{local right noeth domain example}.
I also thank W.~Keith Nicholson for directing me to the reference~\cite{Osofsky} and for
suggesting Definition~\ref{closure under summands} in order to clarify certain arguments
in~\S\ref{essential right ideal section}.
Finally, I thank the referee for providing some helpful references.

\bibliographystyle{amsplain}
\bibliography{NoncommCohenKap.v2}

\def\cprime{$'$}
\providecommand{\bysame}{\leavevmode\hbox to3em{\hrulefill}\thinspace}
\providecommand{\MR}{\relax\ifhmode\unskip\space\fi MR }
\providecommand{\MRhref}[2]{%
  \href{http://www.ams.org/mathscinet-getitem?mr=#1}{#2}
}
\providecommand{\href}[2]{#2}
\begin{thebibliography}{10}

\bibitem{AmitsurSmall}
S.~A. Amitsur and L.~W. Small, \emph{Finite-dimensional representations of {PI}
  algebras}, J. Algebra \textbf{133} (1990), no.~2, 244--248. \MR{1067405
  (91h:16029)}

\bibitem{BeachyWeakley}
John~A. Beachy and William~D. Weakley, \emph{A note on prime ideals which test
  injectivity}, Comm. Algebra \textbf{15} (1987), no.~3, 471--478. \MR{882795
  (88f:16026)}

\bibitem{Bhatwadekar}
S.~M. Bhatwadekar, \emph{On the global dimension of some filtered algebras}, J.
  London Math. Soc. (2) \textbf{13} (1976), no.~2, 239--248. \MR{0404398 (53
  \#8200)}

\bibitem{Chandran1}
V.~R. Chandran, \emph{On two analogues of {C}ohen's theorem}, Indian J. Pure
  Appl. Math. \textbf{8} (1977), no.~1, 54--59. \MR{0453809 (56 \#12062)}

\bibitem{Chandran2}
\bysame, \emph{On two analogues of {C}ohen's theorem}, Pure Appl. Math. Sci.
  \textbf{7} (1978), no.~1-2, 5--10. \MR{0460378 (57 \#372)}

\bibitem{Cohen}
I.~S. Cohen, \emph{Commutative rings with restricted minimum condition}, Duke
  Math. J. \textbf{17} (1950), 27--42. \MR{0033276 (11,413g)}

\bibitem{Cohn}
P.~M. Cohn, \emph{Free {I}deal {R}ings and {L}ocalization in {G}eneral
  {R}ings}, New Mathematical Monographs, vol.~3, Cambridge University Press,
  Cambridge, 2006. \MR{2246388 (2007k:16020)}

\bibitem{Eisenbud}
David Eisenbud, \emph{Commutative {A}lgebra: with a {V}iew {T}oward {A}lgebraic
  {G}eometry}, Graduate Texts in Mathematics, vol. 150, Springer-Verlag, New
  York, 1995. \MR{1322960 (97a:13001)}

\bibitem{Evans73}
E.~Graham Evans, Jr., \emph{Krull-{S}chmidt and cancellation over local rings},
  Pacific J. Math. \textbf{46} (1973), 115--121. \MR{0323815 (48 \#2170)}

\bibitem{GoldiePIR}
A.~W. Goldie, \emph{Non-commutative principal ideal rings}, Arch. Math.
  \textbf{13} (1962), 213--221. \MR{0140532 (25 \#3951)}

\bibitem{Goldie}
\bysame, \emph{Properties of the idealiser}, Ring theory ({P}roc. {C}onf.,
  {P}ark {C}ity, {U}tah, 1971), Academic Press, New York, 1972, pp.~161--169.
  \MR{0382341 (52 \#3226)}

\bibitem{Goodearl}
K.~R. Goodearl, \emph{Global dimension of differential operator rings. {II}},
  Trans. Amer. Math. Soc. \textbf{209} (1975), 65--85. \MR{0382359 (52 \#3244)}

\bibitem{GoodearlWarfield}
K.~R. Goodearl and R.~B. Warfield, Jr., \emph{An {I}ntroduction to
  {N}oncommutative {N}oetherian {R}ings}, second ed., London Mathematical
  Society Student Texts, vol.~61, Cambridge University Press, Cambridge, 2004.
  \MR{2080008 (2005b:16001)}

\bibitem{GordonRobson}
Robert Gordon and J.~C. Robson, \emph{Krull {D}imension}, American Mathematical
  Society, Providence, R.I., 1973, Memoirs of the American Mathematical
  Society, No. 133. \MR{0352177 (50 \#4664)}

\bibitem{Huynh}
Dinh~Van Huynh, \emph{A note on rings with chain conditions}, Acta Math.
  Hungar. \textbf{51} (1988), no.~1-2, 65--70. \MR{934584 (89e:16024)}

\bibitem{Jategaonkar1}
Arun~Vinayak Jategaonkar, \emph{Left principal ideal domains}, J. Algebra
  \textbf{8} (1968), 148--155. \MR{0218387 (36 \#1474)}

\bibitem{Jategaonkar}
\bysame, \emph{A counter-example in ring theory and homological algebra}, J.
  Algebra \textbf{12} (1969), 418--440. \MR{0240131 (39 \#1485)}

\bibitem{KapDivisors}
Irving Kaplansky, \emph{Elementary divisors and modules}, Trans. Amer. Math.
  Soc. \textbf{66} (1949), 464--491. \MR{0031470 (11,155b)}

\bibitem{Kertesz}
A.~Kert{\'e}sz, \emph{Noethersche {R}inge, die artinsch sind}, Acta Sci. Math.
  (Szeged) \textbf{31} (1970), 219--221. \MR{0279126 (43 \#4852)}

\bibitem{Koh1}
Kwangil Koh, \emph{On one sided ideals of a prime type}, Proc. Amer. Math. Soc.
  \textbf{28} (1971), 321--329. \MR{0274488 (43 \#251)}

\bibitem{Koh}
\bysame, \emph{On prime one-sided ideals}, Canad. Math. Bull. \textbf{14}
  (1971), 259--260. \MR{0313325 (47 \#1880)}

\bibitem{Koker}
John~J. Koker, \emph{Global dimension of rings with {K}rull dimension}, Comm.
  Algebra \textbf{20} (1992), no.~10, 2863--2876. \MR{1179266 (94a:16011)}

\bibitem{Krause}
G{\"u}nter Krause, \emph{On fully left bounded left noetherian rings}, J.
  Algebra \textbf{23} (1972), 88--99. \MR{0308188 (46 \#7303)}

\bibitem{Lectures}
T.~Y. Lam, \emph{Lectures on {M}odules and {R}ings}, Graduate Texts in
  Mathematics, vol. 189, Springer-Verlag, New York, 1999. \MR{1653294
  (99i:16001)}

\bibitem{FC}
\bysame, \emph{A {F}irst {C}ourse in {N}oncommutative {R}ings}, second ed.,
  Graduate Texts in Mathematics, vol. 131, Springer-Verlag, New York, 2001.
  \MR{1838439 (2002c:16001)}

\bibitem{ExercisesClassical}
\bysame, \emph{Exercises in {C}lassical {R}ing {T}heory}, second ed., Problem
  Books in Mathematics, Springer-Verlag, New York, 2003. \MR{2003255
  (2004g:16001)}

\bibitem{CrashCourse}
\bysame, \emph{A crash course on stable range, cancellation, substitution and
  exchange}, J. Algebra Appl. \textbf{3} (2004), no.~3, 301--343. \MR{2096452
  (2005g:16007)}

\bibitem{ExercisesModules}
\bysame, \emph{Exercises in {M}odules and {R}ings}, Problem Books in
  Mathematics, Springer, New York, 2007. \MR{2278849 (2007h:16001)}

\bibitem{LR}
T.~Y. Lam and Manuel~L. Reyes, \emph{A {P}rime {I}deal {P}rinciple in
  commutative algebra}, J. Algebra \textbf{319} (2008), no.~7, 3006--3027.
  \MR{2397420 (2009c:13003)}

\bibitem{McConnellRobson}
J.~C. McConnell and J.~C. Robson, \emph{Noncommutative {N}oetherian {R}ings},
  revised ed., Graduate Studies in Mathematics, vol.~30, American Mathematical
  Society, Providence, RI, 2001. \MR{1811901 (2001i:16039)}

\bibitem{Michler}
Gerhard~O. Michler, \emph{Prime right ideals and right noetherian rings}, Ring
  theory ({P}roc. {C}onf., {P}ark {C}ity, {U}tah, 1971), Academic Press, New
  York, 1972, pp.~251--255. \MR{0340334 (49 \#5089)}

\bibitem{Osofsky}
B.~L. Osofsky, \emph{A generalization of quasi-{F}robenius rings}, J. Algebra
  \textbf{4} (1966), 373--387. \MR{0204463 (34 \#4305)}

\bibitem{Reyes}
Manuel~L. Reyes, \emph{A one-sided {P}rime {I}deal {Principle} for
  noncommutative rings}, J. Algebra Appl. \textbf{9} (2010), no.~6, 877--919.

\bibitem{RobsonPrincipal}
J.~C. Robson, \emph{Rings in which finitely generated right ideals are
  principal}, Proc. London Math. Soc. (3) \textbf{17} (1967), 617--628.
  \MR{0217109 (36 \#200)}

\bibitem{Robson}
\bysame, \emph{Decomposition of {N}oetherian rings}, Comm. Algebra \textbf{1}
  (1974), 345--349. \MR{0342564 (49 \#7310)}

\bibitem{Rotman}
Joseph~J. Rotman, \emph{An {I}ntroduction to {H}omological {A}lgebra}, Pure and
  Applied Mathematics, vol.~85, Academic Press Inc. [Harcourt Brace Jovanovich
  Publishers], New York, 1979. \MR{538169 (80k:18001)}

\bibitem{Smith1}
P.~F. Smith, \emph{Injective modules and prime ideals}, Comm. Algebra
  \textbf{9} (1981), no.~9, 989--999. \MR{614468 (82h:16018)}

\bibitem{Smith2}
\bysame, \emph{The injective test lemma in fully bounded rings}, Comm. Algebra
  \textbf{9} (1981), no.~17, 1701--1708. \MR{631883 (82k:16030)}

\bibitem{Smith3}
\bysame, \emph{Concerning a theorem of {I}. {S}. {C}ohen}, An. \c Stiin\c t.
  Univ. Ovidius Constan\c ta Ser. Mat. \textbf{2} (1994), 160--167, XIth
  National Conference of Algebra (Constan{\c{t}}a, 1994). \MR{1367558
  (96m:16029)}

\bibitem{Zabavskii}
B.~V. Zabavs{\cprime}ki{\u\i}, \emph{A noncommutative analogue of {C}ohen's
  theorem}, Ukra\"\i n. Mat. Zh. \textbf{48} (1996), no.~5, 707--710.
  \MR{1417038}

\end{thebibliography}
\end{document}